\documentclass[11pt,letterpaper,reqno]{amsart}
\usepackage[left=30mm,top=40mm,right=30mm,bottom=35mm]{geometry}
\usepackage{amsfonts}
\usepackage{amsmath}
\usepackage{amssymb}
\usepackage{amsthm}
\usepackage{times}
\usepackage{color}
\usepackage{enumerate}
\usepackage[T1]{fontenc}
\usepackage{inputenc}
\usepackage{graphicx}
\usepackage{xcolor}
\usepackage{comment}

\usepackage[textwidth=30mm]{todonotes}

\usepackage{hyperref}
\hypersetup{
 colorlinks=true,
 linkcolor=blue,
 citecolor=blue,
 filecolor=blue,
 urlcolor=blue
}
\usepackage[nameinlink,capitalize,noabbrev]{cleveref}
\usepackage{aliascnt}
\usepackage{mathrsfs}
\usepackage{booktabs}
\usepackage{tikz}
\usepackage{stmaryrd}
\usepackage{float}

\usepackage[normalem]{ulem}

\newtheorem{theorem}{Theorem}[section]
\newtheorem*{theorem*}{Theorem}

\newaliascnt{corollary}{theorem}

\aliascntresetthe{corollary}
\crefname{corollary}{Corollary}{Corollaries}
\Crefname{corollary}{Corollary}{Corollaries}

\newaliascnt{lemma}{theorem}
\newtheorem{lemma}[lemma]{Lemma}
\aliascntresetthe{lemma}
\crefname{lemma}{Lemma}{Lemmas}
\Crefname{lemma}{Lemma}{Lemmas}

\newaliascnt{proposition}{theorem}
\newtheorem{proposition}[proposition]{Proposition}
\aliascntresetthe{proposition}
\crefname{proposition}{Proposition}{Propositions}
\Crefname{proposition}{Proposition}{Propositions}

\newaliascnt{problem}{theorem}

\aliascntresetthe{problem}
\crefname{problem}{Problem}{Problems}
\Crefname{problem}{Problem}{Problems}

\newaliascnt{conjecture}{theorem}

\aliascntresetthe{conjecture}
\crefname{conjecture}{Conjecture}{Conjectures}
\Crefname{conjecture}{Conjecture}{Conjectures}

\theoremstyle{remark}
\newaliascnt{remark}{theorem}
\newtheorem{remark}[remark]{Remark}
\aliascntresetthe{remark}
\crefname{remark}{Remark}{Remarks}
\Crefname{remark}{Remark}{Remarks}

\theoremstyle{definition}
\newaliascnt{definition}{theorem}
\newtheorem{definition}[definition]{Definition}
\aliascntresetthe{definition}
\crefname{definition}{Definition}{Definitions}
\Crefname{definition}{Definition}{Definitions}

\newaliascnt{notation}{theorem}

\aliascntresetthe{notation}
\crefname{notation}{Notation}{Notations}
\Crefname{notation}{Notation}{Notations}

\newaliascnt{example}{theorem}
\newtheorem{example}[example]{Example}
\aliascntresetthe{example}
\crefname{example}{Example}{Examples}
\Crefname{example}{Example}{Examples}

\newaliascnt{strategy}{theorem}

\aliascntresetthe{strategy}
\crefname{strategy}{Strategy}{Strategies}
\Crefname{strategy}{Strategy}{Strategies}

\newaliascnt{algorithm}{theorem}
\newtheorem{algorithm}[algorithm]{Algorithm}
\aliascntresetthe{algorithm}
\crefname{algorithm}{Algorithm}{Algorithms}
\Crefname{algorithm}{Algorithm}{Algorithms}

\def\M{\mathfrak{m}}
\def\XX{\mathbf{x}}

\newcommand{\KK}[0]{\ensuremath{\mathbf{k}}}
\newcommand{\ZZ}[0]{\ensuremath{\mathbb{Z}}}
\newcommand{\GA}[0]{\ensuremath{\mathbb{G}_{\mathrm{a}}}}
\newcommand{\RR}[0]{\ensuremath{\mathbb{R}}}

\newcommand{\spec}[0]{\ensuremath{\operatorname{Spec}}}
\newcommand{\supp}[0]{\ensuremath{\operatorname{supp}}}
\newcommand{\Aut}[0]{\ensuremath{\operatorname{Aut}}}
\newcommand{\Sym}[0]{\ensuremath{\operatorname{Sym}}}
\newcommand{\GL}[0]{\ensuremath{\operatorname{GL}}}
\newcommand{\Der}[0]{\ensuremath{\operatorname{Der}}}
\newcommand{\R}{\mathcal R}
\newcommand{\jac}[0]{\ensuremath{\operatorname{Jac}}}
\newcommand{\divi}[0]{\ensuremath{\operatorname{div}}}

\newcommand{\quot}[1]{{\overline{#1}}}

\begin{document}

\title[Structure of monomial automorphism groups]{The structure of automorphism groups of zero-dimensional monomial algebras}

\author{Roberto D\'iaz}
\address{Departamento de Matem\'aticas, Facultad de Ciencias, Universidad de La Serena, Juan
Cisternas 1200, La Serena, Chile.}%
\email{roberto.diazv1@userena.cl}

\author{Giancarlo Lucchini Arteche} %
\address{Departamento de Matem\'aticas, Facultad de Ciencias, Universidad de Chile, Las Palmeras 3425, \~{N}u\~{n}oa, Santiago, Chile.}
\email{luco@uchile.cl}

\author{Gonzalo Manzano-Flores} %
\address{Instituto de Matem\'aticas, Facultad de Ciencias, Universidad de Valpara\'iso, Gran Breta\~{n}a 1111, Valpara\'iso, Chile.}
\email{gonzalo.manzano@uv.cl}

\date{\today}

\thanks{
 \mbox{\hspace{11pt}}{\it Key words}: monomial algebras, monomial ideals, automorphism groups, tame automorphisms.}

\begin{abstract}
Let $A$ be a zero-dimensional monomial algebra over an algebraically closed field of
characteristic zero, that is, a finite-dimensional quotient of a polynomial ring by a
monomial ideal. Its automorphism group $G$ is a linear algebraic group, described through the
homogeneous nilpotent derivations of $A$. We analyze the structure of $G$ in
detail. Its identity component $G^0$ is a semidirect product of its unipotent radical and a reductive subgroup isomorphic to a product of general linear groups, and for each root degree we characterize when the associated derivations give rise to an additive root subgroup, and determine its dimension. Using the Lie brackets of these
derivations, we then give an explicit algorithm that produces, out of the minimal monomial generators of the ideal, a family of root subgroups generating $G^0$ together with a maximal torus. Such a family is minimal in the generic case.
We also show that the component group $G/G^0$ can be arbitrary: every finite group
arises as the component group of the automorphism group of some zero-dimensional monomial algebra. Finally, we apply these results to the algebras $\KK[\XX]/\M^d$, showing that the subgroup generated by a maximal torus and the outer root subgroups is exactly the subgroup of automorphisms with constant Jacobian determinant, and we deduce from this a new proof of Anick's theorem on the density of the tame automorphisms of $\KK[\XX]$.
\end{abstract}
\maketitle

%\tableofcontents

\section*{Introduction}
Let $\KK$ be an algebraically closed field of characteristic zero and let
$\KK[\XX]=\KK[x_1,\dots,x_n]$ be a polynomial ring. A monomial ideal
$I\subseteq\KK[\XX]$ is an ideal admitting a generating set consisting of
monomials, and the quotient $A=\KK[\XX]/I$ is called a monomial algebra.
We are interested in the case where $A$ is zero-dimensional, that is,
finite-dimensional as a $\KK$-vector space; combinatorially this means that
only finitely many monomials survive in $A$. The goal of this paper is to describe, as explicitly as possible, the linear algebraic group $G=\Aut_\KK(A)$ of such an algebra.

Our starting point is \cite{DLMR24}, where the structure of $G$ is analyzed through the homogeneous locally nilpotent derivations of $A$ \cite[Theorem~4.21]{DLMR24}. This point of view goes back
to Demazure's study of the Cremona group \cite[Section~2, Th\'eor\`eme~2]{demazure1970sous} and its modern
reformulation for toric varieties through Demazure roots
\cite[Theorem~2.7]{liendo2010affine}, \cite[Thm.~2]{LL21}: every homogeneous locally nilpotent derivation gives rise to a $\GA$-action on $A$, and the identity component
$G^0$ is generated by the maximal torus $T$ together with the associated root
subgroups. The homogeneous derivations of $A$ split into inner ones (whose degree lies in $\ZZ^n_{\geq 0}$) and outer ones (whose degree has a single negative coordinate, equal to $-1$). The present work takes this classification as an
input and studies the structure of $G$ in three directions, one per section below,
before applying the results, in a last section, to recover a theorem of Anick on tame
automorphisms.

Concerning this last application, it is worth mentioning that the generation problem for the automorphism group of the polynomial ring $\KK[\XX]$ has a long history of its own (see \cite[\S5.1]{vdE00}). By theorems of Jung \cite{Jung42} and van der Kulk \cite{Kulk53}, $\Aut_\KK(\KK[x_1,x_2])$ is generated by tame automorphisms, that is, affine and elementary automorphisms, whereas in three variables the Nagata automorphism \cite{Nagata72} is wild, as proved by Shestakov and Umirbaev \cite{SU04}, so that no such description is available in general. Anick's theorem \cite{Anick83} is an approximation statement in this circle of ideas: an endomorphism of $\KK[\XX]$ whose Jacobian determinant is a nonzero constant (expected to be an automorphism by the recently disproved Jacobian conjecture) cannot be told apart from a tame automorphism to any prescribed order. Now, ``approximate to order $d$'' means ``equal in $\KK[\XX]/\M^d$'', and this is a zero-dimensional monomial algebra. Our study of their automorphism group unveils that the automorphisms of $\KK[\XX]/\M^d$ coming from tame automorphisms of $\KK[\XX]$ are those assembled from root subgroups: the outer root subgroups lift to elementary automorphisms and the torus lifts to diagonal ones. The structure theory developed below identifies the subgroup they generate, and the answer is that it is precisely the subgroup of automorphisms with constant Jacobian determinant.\\

We now give a brief summary of each section of the article.

\medskip\noindent\textbf{General structure.} \cref{sec:structure} is devoted to
the general structure of $G$. We begin by studying the root subgroups themselves. For a
root degree $\alpha$, the homogeneous derivations of degree $\alpha$ are parametrized by a
$\KK$-subspace of covectors $N(\alpha)\subseteq N_\KK$ (\cref{def:roots}), and
exponentiating them yields a map $\iota_\alpha\colon N(\alpha)\to G$,
$p\mapsto\exp(\overline\partial_{\alpha,p})$, which we show to be always injective (\cref{lem:iota-injective}). In
contrast with the toric case, where every such family is one-dimensional
\cite[Th\'eor\`eme~2\,(b)]{demazure1970sous}, the image of $\iota_\alpha$ may here be a
higher-dimensional additive group: we characterize, in combinatorial terms involving
$\alpha$ and $\supp(I)$, exactly when $\iota_\alpha$ is a group homomorphism, in which
case its image is a generalized root subgroup of $G$ isomorphic to $\GA^m$ with
$m=\dim N(\alpha)$ (\cref{prop:gen-root}). This is automatic for outer degrees, but it may fail for inner ones.

With these subgroups at hand, we then use the
natural action of $G$ on the cotangent space $\M/\M^2$ to decompose $G$ as a
semidirect product $G\cong U_1\rtimes G_1$, where $U_1$ is a connected unipotent
normal subgroup, consisting of the automorphisms that act trivially on
$\M/\M^2$, and $G_1$ is generated by $T$, the toric automorphisms and the root
subgroups of degree-sum zero (\cref{prop:semidirect}). We further present the Levi decomposition of $G^0$, identifying its reductive part:
up to a permutation of the variables it is a standard Levi subgroup of
$\GL_n(\KK)$, and hence isomorphic to a product $\prod_{j=1}^r\GL_{m_j}(\KK)$
with $m_1+\cdots+m_r=n$ (\cref{prop:levi}). We also show that every such Levi subgroup is realized:
given any standard Levi subgroup $L$ of $\GL_n(\KK)$ (or the corresponding
standard parabolic subgroup $P_L$), we construct a monomial ideal whose automorphism
group has $L$ (resp.~$P_L$) as the identity component of its linear part (\cref{prop:levi exist}).

\medskip\noindent\textbf{The component group.} The component group $G/G^0$ is
finite and generated by toric automorphisms, i.e.\ by permutations of the variables
preserving $I$. In \cref{sec:fin-groups} we determine which finite groups occur as
$G/G^0$: every one does. Given a finite group $\Gamma$, Frucht's theorem
\cite[Sections~2--3]{Frucht1939} provides a finite graph $H$ with $\Aut(H)\cong\Gamma$, and a simple
monomial ideal built from $H$ realizes $\Gamma$ as $G/G^0$; we moreover arrange the
construction so that $G\cong G^0\rtimes\Gamma$ with $G^0$ solvable (\cref{thm:realize-fin}), and we record
explicit families realizing the symmetric and the dihedral groups.

\medskip\noindent\textbf{Generating sets.} \cref{sec:generators}
refines \cref{thm:dlmr} by extracting a minimal set of roots $\gamma$ such that the family of root subgroups $U_{\gamma,p}$ along with the torus $T$, generate the whole group. The right measure is Lie-theoretic:
for a degree $\gamma$ we let $B_\gamma\subseteq\mathfrak g_\gamma$ be the sum of the brackets
$[\mathfrak g_\alpha,\mathfrak g_\beta]$ over all pairs of roots with $\alpha+\beta=\gamma$,
and we call $\gamma$ reducible when $B_\gamma=\mathfrak g_\gamma$
(\cref{def:irreducible-degree}); the codimension $\mu(\gamma)$ of $B_\gamma$ in
$\mathfrak g_\gamma$ is then a lower bound for the number of generators needed in
degree~$\gamma$. Reducibility by itself does not allow one to discard a degree, since the
reductions may be circular: already for $I=\M^2$ all six roots are reducible, yet
discarding them all leaves only the torus (\cref{ex:reducible-not-discardable}).

The two families of roots are then analyzed separately. Every inner root other than the
$e_i$ and the $2e_i$ is reducible (\cref{lem:recovery-mixed,lem:recovery-multiple}); for the
degrees $2e_i$ we give a combinatorial criterion for reducibility and show that
$\mu(2e_i)\leq1$ (\cref{prop:two-ei-criterion}); and for the degrees $e_i$, where no
direction plays a distinguished role, we list the three possible decompositions and read
$B_{e_i}$ off them (\cref{lem:ei-decompositions,prop:ei-Iouter-search}). This reduces the
inner family to a set $E_{\mathrm{min}}$ of degrees that are minimal for the relation
$e_i\succ e_j$ given by $e_i-e_j\in\R(I)$ (\cref{lem:discardable}). On the outer side, the
roots in each direction are computed from the minimal generators of $I$ by iterated colon
ideals, which cuts the family down to an explicit finite set $\widetilde G(I)$
(\cref{rmk:outer-colon,lem:outer-minimal,lem:outer-reducible}); both reductions are combined
in \cref{prop:discardable2}. This proposition is good enough for pinning a generating set that is quickly computable, but we go a bit further by trying to extract from this set those roots that are redundant. \cref{strategy:generators} takes as input the minimal monomial generating set of $I$ and yields as output a family of root subgroups which, together with $T$, generates $G^0$ and is minimal for this property when the Levi subgroup is $L=T$, which is the generic case (\cref{thm:minimal-generators}). In all generality, the family of roots retained is minimal, in the sense that one cannot get rid of all $U_{\alpha,p}$ in the output for a fixed $\alpha$. The algorithm is run in full on four ideals in
\cref{ex:strategy-run}.

\medskip\noindent\textbf{An application: Anick's theorem.} \cref{sec:anick} puts the previous
analysis to work on a concrete family, the algebras $A_d=\KK[\XX]/\M^d$. We first describe
their roots: the inner ones are the $\alpha\in\ZZ^n_{\geq0}$ with $1\leq s(\alpha)\leq d-2$, the roots of degree-sum zero are all the $e_i-e_j$, so that
$L\cong\GL_n(\KK)$ and $G=G^0$, and every inner root is reducible
(\cref{lem:anick-roots}). This is a case where reducibility does not suffice to discard: $E_{\mathrm{min}}$ consists of a single $e_i$ and \cref{strategy:generators} retains $\mathfrak g_{e_i}$ (\cref{rmk:anick-not-discardable}). However, these inner generators are detected by the Jacobian operator. Every class in $A_d$ has a unique representative of degree $<d$, which makes the operator $\jac$ available on $G$ and its infinitesimal counterpart $\divi$ available on $\Der(A_d)$. As outer root derivations have trivial divergence, we use Lie brackets to produce, in each inner root space, the whole hyperplane of derivations with trivial divergence. This implies that the subgroup $H$ generated by the maximal torus and the outer root subgroups is exactly the subgroup $G^{\jac}$ of automorphisms with constant Jacobian determinant, of codimension $\binom{n+d-2}{n}-1$ in $G$ (\cref{prop:anick-outer}). Since the outer root subgroups lift to elementary automorphisms of $\KK[\XX]$ and the torus lifts to diagonal ones, we obtain a new proof of a theorem of Anick \cite[Theorem~1]{Anick83}: an endomorphism of $\KK[\XX]$ whose Jacobian determinant is a nonzero constant coincides, modulo $\M^d$ and for every $d$, with a tame automorphism, that is, the tame automorphisms are dense for the $\M$-adic topology in the set of such endomorphisms (\cref{thm:anick}).

\medskip

\noindent\textbf{Acknowledgements.} The first author was supported by ANID FONDECYT Iniciaci\'on 11260788. The second author was partially supported by ANID FONDECYT Regular 1240001. The third author was supported by ANID Subvenci\'on a la Instalaci\'on en la Academia 85250089 and by ANID FONDECYT Postdoctoral 3240191.

The authors acknowledge the use of Claude AI during the development of this work. This tool was mainly used for the construction of examples and to help with redaction issues. All mathematical arguments, proofs, and final verifications were carried out independently by the authors, who assume full responsibility for the contents of the paper.

\section{Preliminaries}\label{sec:preliminaries}

Let $\KK$ be an algebraically closed field of characteristic zero and let $n\in\ZZ_{>0}$. Throughout this paper, we fix a free $\ZZ$-module $M$ of rank $n$. We also define $N$ as the dual $\ZZ$-module $N=\operatorname{Hom}(M,\ZZ)$. There is a natural duality pairing
$$\langle\ ,\ \rangle \colon M\times N\to \ZZ,\quad  \mbox{defined by}\quad \langle m,p\rangle:=p(m)\,.$$
We also introduce the $\KK$-vector spaces $M_\KK=M\otimes_\ZZ \KK\simeq \KK^n$ and $N_\KK=N\otimes_\ZZ \KK\simeq \KK^n$. The duality pairing above naturally extends to a pairing $\langle\ ,\ \rangle \colon M_\KK\times N_\KK\to \KK$. If we identify $M=\ZZ^n$ via the canonical basis $E=\{e_1,\ldots,e_n\}$ and $N=\ZZ^n$ via the dual basis $E^*=\{e^*_1,\ldots,e^*_n\}$ of $E$, then the duality pairing simply becomes the standard scalar product.

\subsection{Semigroup algebras and monomial ideals}
The semigroup algebra $\KK[S]$ of a monoid $S$ is defined as
$$\KK[S]=\bigoplus_{m\in S}\KK\cdot\XX^m\quad\mbox{where}\quad \XX^m\cdot\XX^{m'}=\XX^{m+m'}\quad\mbox{and}\quad \XX^0=1\,.$$
If $S$ is the monoid generated by $E$, then $S=\ZZ_{\geq 0}^n$. Setting $\XX^{e_i}=x_i$ induces an isomorphism between the semigroup algebra $\KK[S]$ and the polynomial ring $\KK[\XX]$, where $\XX^{m}$ represents the monomial $x_1^{m_1}\cdots x_n^{m_n}$. We will assume that $S=\ZZ_{\geq 0}^n$ throughout this paper.

An ideal $I\subset \KK[\XX]$ is called monomial if it admits a generating set consisting of monomials, namely, $I=(\XX^{\mathbf{a}_1},\dots,\XX^{\mathbf{a}_l})$, with $\mathbf{a}_k\in \ZZ^n_{\geq 0}$. The support of a monomial ideal $I=(\XX^{\mathbf{a}_1},\dots,\XX^{\mathbf{a}_l})$ is the set $$\supp(I)=\{m\in \ZZ^n_{\geq 0}\mid \XX^m\in I\}=\bigcup_{k=1}^l (\mathbf{a}_k+\ZZ^n_{\geq 0})\,.$$
Following \cite[Section~1]{DLMR24}, we say that a monomial ideal $I=(\XX^{\mathbf{a}_1},\dots,\XX^{\mathbf{a}_l})\subset \KK[\XX]$ has cofinite support if $\ZZ_{\geq 0}^n\setminus\supp(I)$ is finite. Since the classes $\quot{\XX}^m$ with $m\in\ZZ^n_{\geq 0}\setminus\supp(I)$ form a $\KK$-basis of $\KK[\XX]/I$, this occurs if and only if $\dim_\KK (\KK[\XX]/I)<\infty$.

The algebra $\KK[\XX]$ is naturally $\ZZ^n_{\geq 0}$-graded and under this grading every monomial ideal $I$ is a graded ideal. Therefore, the quotient $\KK[\XX]/I$ inherits a natural $\ZZ^n_{\geq 0}$-grading given by $\deg \quot{\XX}^m=m$ if $m\notin \supp(I)$ (recall that if $m\in \supp(I)$, then $\quot{\XX}^m=0$). In the sequel, we always regard $\KK[\XX]/I$ as a $\ZZ^n_{\geq 0}$-graded algebra. Throughout this paper, we also make the assumption that $x_i=\XX^{e_i}\not\in\supp(I)$. This is harmless since any such algebra is also a quotient of a polynomial algebra (in fewer variables) by a monomial ideal.

\subsection{Derivations}\label{sec:derivations}
A derivation $\partial$ on a $\KK$-algebra $B$ is a linear $\KK$-map satisfying the Leibniz rule, i.e., 
$$
\partial(fg)=\partial(f)g+f\partial(g), \quad \text{for all } f,g \in B.
$$
We denote by $\Der(B)$ the vector space of the derivations of $B$. Let $I\subset B$ be an ideal. We also denote by $\Der_I(B)\subset \Der(B)$ the subspace of derivations $\partial$ such that $\partial(I)\subset I$. Defining $[\partial,\partial']:=\partial\circ\partial'-\partial'\circ\partial$, both spaces are equipped with a Lie algebra structure.

A derivation $\partial$ is locally nilpotent if for all $f\in B$ there exists $\ell$ such that $\partial^\ell(f)=0,$ where $\partial^\ell$ is the composition of $\partial$ $\ell$ times. Note that locally nilpotent and nilpotent are equivalent in the case of a finite-dimensional algebra.

Let now $I$ be a monomial ideal on $\KK[\XX]$ and let $\partial\colon\KK[\XX]/I\to\KK[\XX]/I$ be a derivation. The Lie algebras $\Der(\KK[S])$ and $\Der_I(\KK[S])$, where $I$ is a monomial ideal, were studied for various semigroups by several authors; see, for example, \cite[Theorem~2.2.1]{B95}, \cite[Theorem~2.2]{T09}, \cite[Proposition~3.1]{KLL15}, \cite[Propositions~1.3 and~1.4]{DLMR24}. We recall here some of these results.

Recall that $\KK[\XX]/I$ is $\ZZ^n_{\geq 0}$-graded. We say that $\partial$ is homogeneous if it sends homogeneous elements to homogeneous elements. By \cite[Lemma~1.1]{DL24} there exists a unique element $\alpha\in M=\ZZ^n$, called the degree of $\partial$, such that for every $\XX^m\notin\ker\partial$ we have $\partial(\XX^m)=\lambda\XX^{m+\alpha}$ for some $\lambda\in \KK^*$. We say that a homogeneous derivation $\partial$ is inner if $\deg\partial\in \ZZ^n_{\geq 0}$ and outer if $\deg\partial\in \ZZ^n\setminus \ZZ^n_{\geq 0}$. By \cite[Theorem~2.2]{DLMR24}, we know that every such derivation comes from a derivation in $\Der_I(\KK[\XX])\subset\Der(\KK[\XX])$. Now, homogeneous derivations of $\KK[\XX]$ have a very particular shape. By \cite[Proposition~3.1]{KLL15}, every homogeneous derivation of degree $\alpha$ of $\KK[\XX]$ has the following form for a unique $p\in N_\KK$:
\begin{align}\label{eq:homogeneous}
\partial_{\alpha,p}\colon \KK[\XX]\to \KK[\XX]\quad\mbox{given by}\quad \XX^m\mapsto p(m)\cdot\XX^{m+\alpha}
\end{align}
In terms of the standard generators of $\Der(\KK[\XX])$ this reads
$$\partial_{\alpha,p}=\XX^{\alpha}\sum_{i=1}^{n}p(e_i)\,x_i\,\frac{\partial}{\partial x_i}\,,
\qquad\mbox{so that in particular}\qquad
\partial_{\alpha,e_i^*}=\XX^{\alpha+e_i}\,\frac{\partial}{\partial x_i}\,.$$
Throughout the paper we use the notation $\partial_{\alpha,p}$ rather than these
explicit expressions.

If $\alpha$ is inner, then every derivation of $\KK[\XX]$ preserves any monomial ideal. In particular, the above formula defines a derivation in $\KK[\XX]/I$ for every $p\in N_\KK$. On the other hand, by \cite[Theorem~2.7]{liendo2010affine} and \cite[Theorem~2.2.1]{B95}, if $\alpha$ is outer then the derivation of $\KK[\XX]$ inducing it must be of the form $\partial_{\alpha,\lambda e^*_k}$ with $1\leq k\leq n$, $\lambda\in \KK$ and $\alpha$ such that:
\begin{enumerate}[(i)]
    \item $e_k^*(\alpha)=-1$;\label{eq:outer-i}
    \item $e_j^*(\alpha)\geq 0$ for every $j\neq k$;\label{eq:outer-ii}
    \item for every $\beta\in\supp(I)$, either $\alpha+\beta\in\supp(I)$ or $\alpha+\beta\not\in\ZZ_{\geq 0}^n$.\label{eq:outer-iii}
\end{enumerate}
Note that the last condition, which ensures that the derivation preserves the ideal, can be checked by only looking at the generators of the ideal $I$.

In order to decide whether such a derivation is nontrivial, we introduce
the following notation. Following \cite{DLA26}, for $\alpha\in\ZZ^n$, define
$$
E_\alpha:=\{e_i\in E\mid \alpha+e_i\in\ZZ_{\geq 0}^{n}\setminus\supp(I)\}.
$$
We have then the following result, which is essentially \cite[Lem.~3.1]{DLA26}.

\begin{lemma}\label{lem:nontrivial-criterion}
Let $\partial_{\alpha,p}\in\Der_I(\KK[\XX])$ be a homogeneous derivation, with $\alpha\in\ZZ^{n}$ and $p\in N_\KK$.
The induced derivation $\overline\partial_{\alpha,p}$ of $\KK[\XX]/I$ is
trivial if and only if $p(e_i)=0$ for every $e_i\in E_\alpha$. In particular, $\overline\partial_{\alpha,e_i^*}$ is nontrivial if and only if $e_i\in E_\alpha$.
\end{lemma}

\begin{proof}
Since $\overline\partial_{\alpha,p}(\overline\XX^m)=p(m)\,\overline\XX^{m+\alpha}$,
the derivation is trivial if and only if $p(m)=0$ or
$m+\alpha\in\supp(I)$ for every $m\in\ZZ_{\geq 0}^n\setminus\supp(I)$.
If $e_i\in E_\alpha$ and $p(e_i)\neq0,$ for some $i,$ then
$\overline\partial_{\alpha,p}(\overline\XX^{e_i})
=p(e_i)\,\overline\XX^{e_i+\alpha}\neq0,$
since $\alpha+e_i\notin\supp(I)$, so $\overline\partial_{\alpha,p}$ is nontrivial.

Conversely, suppose $p(e_i)=0$ for every $e_i\in E_\alpha$.
For any $m\in\ZZ_{\geq 0}^n$ with
$p(m)=\sum_i p(e_i)\,m_i\neq0$, there exists $i$ with $p(e_i)\neq0$ and
$m_i\geq1$, hence $e_i\notin E_\alpha$ and $\alpha+e_i\in\supp(I)$.
Since $m\geq e_i$, we get
$m+\alpha\in\alpha+e_i+\ZZ_{\geq 0}^n\subset\supp(I)$,
so $\overline\XX^{m+\alpha}=0$ in $\KK[\XX]/I$.
\end{proof}

\begin{definition}\label{def:roots}
We set $\mathcal R(I)$ as the set of degrees $\alpha\neq0$ such that there exists a nontrivial derivation of $\KK[\XX]/I$ of degree $\alpha$. We call them roots, and we distinguish between inner and outer roots in the same way as for degrees. By \cref{lem:nontrivial-criterion}, an inner degree $\alpha$ is in $\mathcal R(I)$ if and only if $E_\alpha$ is nonempty, while an outer degree $\alpha$ is in $\mathcal R(I)$ if and only if it satisfies \eqref{eq:outer-iii} and $E_\alpha$ is nonempty.

Since for outer roots $E_\alpha$ corresponds to a unique element $e_i$ of the canonical basis, the set of outer roots is the disjoint union of the sets $\mathcal{R}_{i}(I):=\{\alpha\text{ outer root}\mid e_i\in E_\alpha\}$. The set $\mathcal R_i(I)$ can also be seen as the set of outer roots whose (unique) negative coordinate is the $i$-th one. Alternatively, we will say that the direction of an outer root $\alpha$ is $e_i^*$ to signal that $\alpha\in\mathcal{R}_{i}(I)$ (since then the only possible derivation, up to scalar, is $\overline\partial_{\alpha,e_i^*}$).

Finally, for $\alpha\in \mathcal{R}(I)$, define
$$
N(\alpha) =
\begin{cases}
  \displaystyle{\bigoplus_{{e_i}\in E_\alpha}} \KK\cdot {e_i^*}  , & \text{if } \alpha \text{ is inner}, \\[0.8em]
 \KK\cdot e_j^* , & \text{if } \alpha \text{ is outer and } \alpha\in\mathcal{R}_{j}(I). \\[0.8em]
\end{cases}
$$
\end{definition}

As a direct consequence of \cref{lem:nontrivial-criterion}, the set $N(\alpha)$ parametrizes the nontrivial derivations of degree $\alpha$. For further reference, we state this as follows.

\begin{lemma}\label{lem:nonzero-deriv}
If $\alpha \in \mathcal{R}(I)$ and $p \in N(\alpha)$, then
$\overline{\partial}_{\alpha,p}=0$ if and only if $p=0$.\qed
\end{lemma}

We close this subsection by recording the Lie bracket of two homogeneous
derivations of $\KK[\XX]/I$, which will be used systematically in the sequel.

\begin{remark}\label{rmk:lie-bracket}
For homogeneous derivations $\overline\partial_{\alpha,p}$ and
$\overline\partial_{\beta,q}$ of $\KK[\XX]/I$,
$$
[\overline\partial_{\alpha,p},\,\overline\partial_{\beta,q}]
= \overline\partial_{\alpha+\beta,\,r},
\qquad r = p(\beta)\,q - q(\alpha)\,p.
$$
The bracket vanishes if and only if $r_i = 0$ for every
$e_i\in E_{\alpha+\beta}$ (\cref{lem:nontrivial-criterion}); in particular it vanishes
whenever $\alpha + \beta \in \supp(I)$.
\end{remark}

\begin{example}
In the following illustration, the left diagram marks as red points the degrees that are ``candidates'' for outer roots, as they satisfy conditions \eqref{eq:outer-i} -- \eqref{eq:outer-iii}, and as green points the degrees that are ``candidates'' for inner roots. However, certain points yield $E_\alpha=\varnothing$, hence in the right diagram we only retain the actual inner and outer roots.

$$
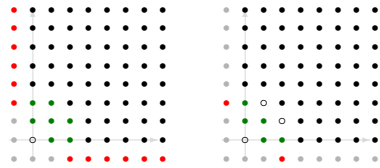
\begin{array}{cc}
\scalebox{0.7}{\begin{picture}(100,80)
\definecolor{gray1}{gray}{0.7}
\definecolor{gray2}{gray}{0.85}
\definecolor{green}{RGB}{0,124,0}

\textcolor{gray2}{\put(0,20){\vector(1,0){80}}}

\textcolor{gray2}{\put(10,10){\vector(0,1){80}}}

\put(0,10){\textcolor{gray1}{\circle*{3}}}
\put(10,10){\textcolor{gray1}{\circle*{3}}}
\put(20,10){\textcolor{gray1}{\circle*{3}}} 
\put(30,10){\textcolor{red}{\circle*{3}}} 
\put(40,10){\textcolor{red}{\circle*{3}}} 
\put(50,10){\textcolor{red}{\circle*{3}}} 
\put(60,10){\textcolor{red}{\circle*{3}}} 
\put(70,10){\textcolor{red}{\circle*{3}}} 
\put(80,10){\textcolor{red}{\circle*{3}}}

\put(0,20){\textcolor{gray1}{\circle*{3}}}
\put(10,20){\circle{3}} 
\put(20,20){\textcolor{green}{\circle*{3}}} 
\put(30,20){\textcolor{green}{\circle*{3}}} 
\put(40,20){{\circle*{3}}} 
\put(50,20){{\circle*{3}}} 
\put(60,20){\circle*{3}}
\put(70,20){\circle*{3}}
\put(80,20){\circle*{3}}

\put(0,30){\textcolor{gray1}{\circle*{3}}}
\put(10,30){\textcolor{green}{\circle*{3}}} 
\put(20,30){\textcolor{green}{\circle*{3}}} 
\put(30,30){\textcolor{green}{\circle*{3}}} 
\put(40,30){{\circle*{3}}} 
\put(50,30){{\circle*{3}}} 
\put(60,30){{\circle*{3}}} 
\put(70,30){\circle*{3}}
\put(80,30){\circle*{3}}

\put(0,40){\textcolor{red}{\circle*{3}}}
\put(10,40){\textcolor{green}{\circle*{3}}} 
\put(20,40){\textcolor{green}{\circle*{3}}} 
\put(30,40){{\circle*{3}}} 
\put(40,40){{\circle*{3}}}
\put(50,40){{\circle*{3}}}
\put(60,40){\circle*{3}} 
\put(70,40){\circle*{3}}
\put(80,40){\circle*{3}}

\put(0,50){\textcolor{red}{\circle*{3}}}
\put(10,50){{\circle*{3}}} 
\put(20,50){{\circle*{3}}} 
\put(30,50){{\circle*{3}}} 
\put(40,50){\circle*{3}}
\put(50,50){\circle*{3}}
\put(60,50){\circle*{3}} 
\put(70,50){\circle*{3}} 
\put(80,50){\circle*{3}}

\put(0,60){\textcolor{red}{\circle*{3}}}
\put(10,60){{\circle*{3}}} 
\put(20,60){{\circle*{3}}} 
\put(30,60){{\circle*{3}}} 
\put(40,60){\circle*{3}}
\put(50,60){\circle*{3}}
\put(60,60){\circle*{3}} 
\put(70,60){\circle*{3}}
\put(80,60){\circle*{3}}

\put(0,70){\textcolor{red}{\circle*{3}}}
\put(10,70){{\circle*{3}}}
\put(20,70){{\circle*{3}}}
\put(30,70){{\circle*{3}}}
\put(40,70){{\circle*{3}}}
\put(50,70){\circle*{3}}
\put(60,70){\circle*{3}} 
\put(70,70){\circle*{3}} 
\put(80,70){\circle*{3}}

\put(0,80){\textcolor{red}{\circle*{3}}}
\put(10,80){{\circle*{3}}}
\put(20,80){{\circle*{3}}}
\put(30,80){{\circle*{3}}}
\put(40,80){{\circle*{3}}}
\put(50,80){{\circle*{3}}}
\put(60,80){\circle*{3}} 
\put(70,80){\circle*{3}} 
\put(80,80){\circle*{3}}

\put(0,90){\textcolor{red}{\circle*{3}}}
\put(10,90){{\circle*{3}}}
\put(20,90){{\circle*{3}}}
\put(30,90){{\circle*{3}}}
\put(40,90){{\circle*{3}}}
\put(50,90){{\circle*{3}}}
\put(60,90){{\circle*{3}}}
\put(70,90){\circle*{3}} 
\put(80,90){\circle*{3}}

\end{picture}} &
\scalebox{0.7}{\begin{picture}(100,80)
\definecolor{gray1}{gray}{0.7}
\definecolor{gray2}{gray}{0.85}
\definecolor{green}{RGB}{0,124,0}

\textcolor{gray2}{\put(0,20){\vector(1,0){80}}}

\textcolor{gray2}{\put(10,10){\vector(0,1){80}}}

\put(0,10){\textcolor{gray1}{\circle*{3}}}
\put(10,10){\textcolor{gray1}{\circle*{3}}}
\put(20,10){\textcolor{gray1}{\circle*{3}}} 
\put(30,10){\textcolor{red}{\circle*{3}}} 
\put(40,10){\textcolor{gray1}{\circle*{3}}} 
\put(50,10){\textcolor{gray1}{\circle*{3}}} 
\put(60,10){\textcolor{gray1}{\circle*{3}}} 
\put(70,10){\textcolor{gray1}{\circle*{3}}}  
\put(80,10){\textcolor{gray1}{\circle*{3}}}

\put(0,20){\textcolor{gray1}{\circle*{3}}}
\put(10,20){\circle{3}} 
\put(20,20){\textcolor{green}{\circle*{3}}} 
\put(30,20){\textcolor{green}{\circle*{3}}} 
\put(40,20){{\circle*{3}}} 
\put(50,20){{\circle*{3}}} 
\put(60,20){\circle*{3}}
\put(70,20){\circle*{3}}
\put(80,20){\circle*{3}}

\put(0,30){\textcolor{gray1}{\circle*{3}}}
\put(10,30){\textcolor{green}{\circle*{3}}} 
\put(20,30){\textcolor{green}{\circle*{3}}} 
\put(30,30){\circle{3}} 
\put(40,30){{\circle*{3}}} 
\put(50,30){{\circle*{3}}} 
\put(60,30){{\circle*{3}}} 
\put(70,30){\circle*{3}}
\put(80,30){\circle*{3}}

\put(0,40){\textcolor{red}{\circle*{3}}}
\put(10,40){\textcolor{green}{\circle*{3}}} 
\put(20,40){\circle{3}} 
\put(30,40){{\circle*{3}}} 
\put(40,40){{\circle*{3}}}
\put(50,40){{\circle*{3}}}
\put(60,40){\circle*{3}} 
\put(70,40){\circle*{3}}
\put(80,40){\circle*{3}}

\put(0,50){\textcolor{gray1}{\circle*{3}}} 
\put(10,50){{\circle*{3}}} 
\put(20,50){{\circle*{3}}} 
\put(30,50){{\circle*{3}}} 
\put(40,50){\circle*{3}}
\put(50,50){\circle*{3}}
\put(60,50){\circle*{3}} 
\put(70,50){\circle*{3}} 
\put(80,50){\circle*{3}}

\put(0,60){\textcolor{gray1}{\circle*{3}}} 
\put(10,60){{\circle*{3}}} 
\put(20,60){{\circle*{3}}} 
\put(30,60){{\circle*{3}}} 
\put(40,60){\circle*{3}}
\put(50,60){\circle*{3}}
\put(60,60){\circle*{3}} 
\put(70,60){\circle*{3}}
\put(80,60){\circle*{3}}

\put(0,70){\textcolor{gray1}{\circle*{3}}} 
\put(10,70){{\circle*{3}}}
\put(20,70){{\circle*{3}}}
\put(30,70){{\circle*{3}}}
\put(40,70){{\circle*{3}}}
\put(50,70){\circle*{3}}
\put(60,70){\circle*{3}} 
\put(70,70){\circle*{3}} 
\put(80,70){\circle*{3}}

\put(0,80){\textcolor{gray1}{\circle*{3}}} 
\put(10,80){{\circle*{3}}}
\put(20,80){{\circle*{3}}}
\put(30,80){{\circle*{3}}}
\put(40,80){{\circle*{3}}}
\put(50,80){{\circle*{3}}}
\put(60,80){\circle*{3}} 
\put(70,80){\circle*{3}} 
\put(80,80){\circle*{3}}

\put(0,90){\textcolor{gray1}{\circle*{3}}} 
\put(10,90){{\circle*{3}}}
\put(20,90){{\circle*{3}}}
\put(30,90){{\circle*{3}}}
\put(40,90){{\circle*{3}}}
\put(50,90){{\circle*{3}}}
\put(60,90){{\circle*{3}}}
\put(70,90){\circle*{3}} 
\put(80,90){\circle*{3}} 

\end{picture}}\\
\multicolumn{2}{c}{I=(x^3,x^2y^2,y^3)}
\end{array}
$$
\end{example}

\section{The general structure of the automorphism group}\label{sec:structure}

Let \( I \subset \KK[\XX] \) be a monomial ideal with cofinite support, let $G=\Aut_\KK(\KK[\XX]/I)$ and consider the subtorus $T=\spec(\KK[M])\subset G$. We restate here the main result of \cite[Theorem~4.21]{DLMR24} concerning the structure of $G$.

For $\alpha\in\mathcal R(I)$ and $p\in N(\alpha)$, write $U_{\alpha,p}$
for the closed subgroup of $G$ generated by the one-parameter group
$\{\exp(t\,\overline\partial_{\alpha,p}):t\in\KK\}$; it is isomorphic to $\GA$ when
$\overline\partial_{\alpha,p}\neq0$ and is trivial otherwise. When $N(\alpha)$ is
one-dimensional, for instance when $\alpha$ is an outer root, there is essentially a
single such subgroup, which we then denote simply by $U_\alpha$.

\begin{theorem}[{\cite[Theorem~4.21]{DLMR24}}]\label{thm:dlmr}
Let $I,G,T$ be as above. Then the following hold:
\begin{enumerate}[(i)]
    \item $G$ is linear algebraic and $T$ is a maximal torus of $G$.
    \item The neutral component $G^0$ is generated by $T$ and the $U_{\alpha, p}$ for $\alpha\in\mathcal R(I)$ and $p\in N(\alpha)$.
    \item The finite quotient $G/G^0$ is generated by toric automorphisms.
\end{enumerate}
\end{theorem}

A toric automorphism of $\KK[\XX]/I$ is an automorphism induced by a
permutation $\sigma\in S_n$ of the variables, i.e.\ $\overline{x}_k\mapsto
\overline{x}_{\sigma(k)}$; such a permutation descends to $\KK[\XX]/I$ if and
only if $\sigma$ preserves $I$ \cite[Definition~4.14]{DLMR24}. The toric
automorphisms form a finite subgroup $\Aut(S,I)\subseteq G$.

Note that the Lie algebra $\mathfrak g$ of $G$ is precisely the algebra of derivations $\Der_\KK(\KK[\XX]/I)$. We see then that the homogeneous components of $\Der_\KK(\KK[\XX]/I)$ correspond to the summands in the weight decomposition
$$\mathfrak g=\bigoplus_{\alpha\in\ZZ^n}\mathfrak g_\alpha,$$
given by the adjoint action of the torus $T\subset G$ on the Lie algebra $\mathfrak g$.

\subsection{Generalized root subgroups and additive actions}\label{sec:root-subgroups}

A generalized root subgroup of an algebraic group $G$ with respect to
a subgroup $H$ is a closed subgroup isomorphic to $\GA^m$ for some $m\in\mathbb{N},$ that is normalized
by $H$ through a single character \cite[Definition~7.1]{RV21b}; the case
$m=1$ recovers the classical root subgroups. In this subsection we study, for
each root $\alpha\in\mathcal R(I)$, the map
$$\iota_\alpha\colon N(\alpha)\longrightarrow G=\Aut_\KK(\KK[\XX]/I),\qquad
p\longmapsto\exp(\overline\partial_{\alpha,p})\,,$$
whose restriction to the line $\KK\cdot p$ has image $U_{\alpha,p}$.
We characterize when
its image is a generalized root subgroup isomorphic to $\GA^m$, in terms of
combinatorial conditions on $\alpha$ and $\supp(I)$. When $\iota_\alpha(N(\alpha))\cong\GA^m$, the resulting subgroup defines an additive action of $\GA^m$ on $\KK[\XX]/I$ normalized by the torus.

\begin{lemma}\label{lem:iota-injective}
For every $\alpha \in \mathcal{R}(I)$, the map $\iota_\alpha$ is injective.
\end{lemma}

\begin{proof}
The map $p\mapsto\overline\partial_{\alpha,p}$ is linear and, by {\cref{lem:nonzero-deriv}}, vanishes only when $p=0$; hence it is injective on $N(\alpha)$.
Since $\alpha\in\mathcal R(I)$, the derivations $\overline\partial_{\alpha,p}$
are nilpotent. Thus $\exp$ has a polynomial inverse
given by $\log = \sum_{k=1}^{n} \frac{(-1)^{k+1}}{k}
(\exp(\overline\partial_{\alpha,p})-\operatorname{id})^k$ for $n$ big enough; in particular $\exp$
is injective on the space of such derivations. Therefore
$\iota_\alpha = \exp\circ\,(p\mapsto\overline\partial_{\alpha,p})$ is the
composition of two injective maps.
\end{proof}

\begin{proposition}\label{prop:gen-root}
Fix a root $\alpha\in\mathcal R(I)$. The image $\iota_\alpha(N(\alpha))$ is a generalized root subgroup of $\Aut_\KK(\KK[\XX]/I)$ with respect to the maximal torus $T$ if and only if one of the following holds:
\begin{enumerate}[(i)]
\item $\alpha$ is outer; or
\item $\alpha$ is inner and, if $E^+_\alpha=\{e_k\in E_\alpha,\;\alpha_k\neq0\}$, then either
\begin{enumerate}
  \item[(ii.1)] $E_{2\alpha}=\varnothing$; or
  \item[(ii.2)] $E_{2\alpha}=E^+_\alpha=\{e_{i_0}\}$.
\end{enumerate}
\end{enumerate}
\end{proposition}

\begin{proof}
Under the identification
$p\mapsto\overline\partial_{\alpha,p}$, which is bijective by the definition of
$N(\alpha)$ and {\cref{lem:nonzero-deriv}}, $N(\alpha)$ corresponds to the degree-$\alpha$ part $\mathfrak g_\alpha$ of
$\mathfrak g=\Der(\KK[\XX]/I)$. We first observe that the normalization
condition holds automatically: for $t\in T$ and $p\in N(\alpha)$ one has
$t\circ\overline\partial_{\alpha,p}\circ t^{-1}
=\chi^{\alpha}(t)\,\overline\partial_{\alpha,p}$, since
$\overline\partial_{\alpha,p}$ is homogeneous of degree $\alpha$; applying
$\exp$ yields
$t\,\iota_\alpha(p)\,t^{-1}=\iota_\alpha\!\big(\chi^{\alpha}(t)\,p\big)$, so
$T$ normalizes $\iota_\alpha(N(\alpha))$ acting through the single character
$\chi^\alpha$. It therefore remains to determine when $\iota_\alpha(N(\alpha))$ is
a closed subgroup isomorphic to $\GA^m$. By \cref{rmk:lie-bracket}, for $p,q\in N(\alpha)$, the element $[\overline\partial_{\alpha,p},\overline\partial_{\alpha,q}]=\overline\partial_{2\alpha,\,r}$, with $r=p(\alpha)q-q(\alpha)p$, lies in $\mathfrak g_{2\alpha}$.

If all elements of $\mathfrak g_\alpha$ commute, then, since
$\overline\partial_{\alpha,p+q}=\overline\partial_{\alpha,p}+\overline\partial_{\alpha,q}$,
the map $\iota_\alpha(p)=\exp(\overline\partial_{\alpha,p})$ satisfies
$\iota_\alpha(p)\circ\iota_\alpha(q)=\exp(\overline\partial_{\alpha,p}
+\overline\partial_{\alpha,q})=\iota_\alpha(p+q)$; thus
$\iota_\alpha\colon(N(\alpha),+)\to\Aut_\KK(\KK[\XX]/I)$ is an injective
homomorphism of algebraic groups and its image is the closed subgroup
$\exp(\mathfrak g_\alpha)\cong\GA^m${; it is closed
because $\exp$ is a polynomial isomorphism onto its image, with polynomial
inverse $\log$ as in \cref{lem:iota-injective}}. Conversely, suppose
$[\overline\partial_{\alpha,p},\overline\partial_{\alpha,q}]
=\overline\partial_{2\alpha,r}\neq0$ for some $p,q$. Every element of
$\iota_\alpha(N(\alpha))=\exp(\mathfrak g_\alpha)$ is the exponential of a derivation of
pure degree $\alpha$; on the other hand, since the derivations are nilpotent, a direct computation gives
$$
  \log\!\big(\iota_\alpha(p)\circ\iota_\alpha(q)\big)
  =\overline\partial_{\alpha,p+q}+\tfrac12\,\overline\partial_{2\alpha,r}
   +(\text{terms in degrees }\ge 3\alpha),
$$
whose degree-$2\alpha$ component $\tfrac12\overline\partial_{2\alpha,r}$ is
nonzero. Hence $\iota_\alpha(p)\circ\iota_\alpha(q)$ is not the exponential of
any degree-$\alpha$ derivation, so it does not lie in
$\iota_\alpha(N(\alpha))$. We deduce that the image is not closed under multiplication, and in
particular is not a subgroup isomorphic to $\GA^m$. This proves that
$\iota_\alpha(N(\alpha))\cong\GA^m$ if and only if the
$\overline\partial_{\alpha,p}$ pairwise commute, i.e.\ if and only if
$\overline\partial_{2\alpha,r}=0$ for all $p,q\in N(\alpha)$.

If $\alpha$ is outer, then $N(\alpha)=\KK e_j^*$ for a single $j$, so the derivations commute and the image is $\GA^1$. Assume now that $\alpha$ is inner. Then $N(\alpha)=\bigoplus_{k\in E_\alpha}\KK e_k^*$ and $m=|E_\alpha|$. As $\supp(I)$ is closed under addition by elements of $\ZZ^n_{\geq 0}$, $\alpha+e_i\in\supp(I)$ forces $2\alpha+e_i\in\supp(I)$; thus
$E_{2\alpha}\subseteq E_{\alpha}$. Recall moreover that $p(e_i)=q(e_i)=0$ for every $e_i\not\in E_\alpha$, and that $\overline\partial_{2\alpha,r}=0$ if and only if
$r(e_i)=0$ for every $e_i\in E_{2\alpha}$. This is obviously the case if $E_{2\alpha}=\varnothing$; and if $E_{2\alpha}=E^+_\alpha=\{e_{i_0}\}$, then
\[r(e_{i_0})=p(\alpha)\,q(e_{i_0})-q(\alpha)\,p(e_{i_0})=\alpha_{i_0}\,p(e_{i_0})\,q(e_{i_0})-\alpha_{i_0}\,q(e_{i_0})\,p(e_{i_0})=0,\]
so that we also get commutativity of $\overline{\partial}_{\alpha,p}$ and $\overline{\partial}_{\alpha,q}$ in this case.

It remains to prove that all other cases yield a nontrivial bracket $\overline\partial_{2\alpha,r}$. Now, if $E_{2\alpha}\neq\varnothing$, then $E^+_\alpha\supseteq \{e_j\in E\mid \alpha_j\neq0\}$. Indeed, if $\alpha_j\geq 1$, then $\alpha+e_j\in\supp(I)$ would force $2\alpha=(\alpha+e_j)+(\alpha-e_j)\in\supp(I)$ and then $E_{2\alpha}=\varnothing$. In particular, we have $E_{\alpha}^+\neq\varnothing$ since $\alpha$ is inner and $\neq 0$. Assuming then that we do not have $E_{2\alpha}=E^+_\alpha=\{e_{i_0}\}$, we may choose $e_i\in E_{2\alpha}$ and $e_j\in E^+_\alpha$ with $i\neq j$. Taking $p=e_i^*$ and $q=e_{j}^*$, both in $N(\alpha)$ since $E^+_\alpha\subseteq E_\alpha$ and $E_{2\alpha}\subseteq E_\alpha$, we get
\[r(e_{i})=p(\alpha)\,q(e_{i})-q(\alpha)\,p(e_{i})=\alpha_i\cdot 0-\alpha_{j}\cdot 1=-\alpha_j\neq0,\]
since $e_j\in E^+_\alpha$. As $e_i\in E_{2\alpha}$, this gives $\overline\partial_{2\alpha,r}\neq 0$.
\end{proof}

\begin{example}\label{ex:gen-root-cases}
The diagram displays the roots $\alpha\in\mathcal R(I)$ for
$I=(x_1^4,\,x_2^4,\,x_1^2x_2)$, coloured according to \cref{prop:gen-root}: outer roots in red, inner roots satisfying
Case~(ii.1) in blue, Case~(ii.2) in green, and roots for which
$\iota_\alpha(N(\alpha))\not\cong\GA^m$ in gray. Hollow points are basis
elements that are not inner roots, and filled black points lie in $\supp(I)$.

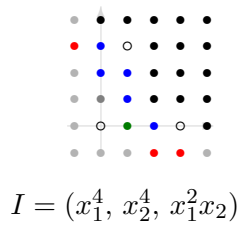
$$
\begin{array}{c}
\begin{picture}(60,70)
\definecolor{gray1}{gray}{0.7}
\definecolor{gray2}{gray}{0.85}
\definecolor{green}{RGB}{0,124,0}

\textcolor{gray2}{\put(0,20){\vector(1,0){55}}}
\textcolor{gray2}{\put(10,10){\vector(0,1){55}}}

% y=-1 row
\put(0,10){\textcolor{gray1}{\circle*{3}}}
\put(10,10){\textcolor{gray1}{\circle*{3}}}
\put(20,10){\textcolor{gray1}{\circle*{3}}}
\put(30,10){\textcolor{red}{\circle*{3}}}
\put(40,10){\textcolor{red}{\circle*{3}}}
\put(50,10){\textcolor{gray1}{\circle*{3}}}

% y=0 row
\put(0,20){\textcolor{gray1}{\circle*{3}}}
\put(10,20){\circle{3}}
\put(20,20){\textcolor{green}{\circle*{3}}}
\put(30,20){\textcolor{blue}{\circle*{3}}}
\put(40,20){\circle{3}}
\put(50,20){\circle*{3}}

% y=1 row
\put(0,30){\textcolor{gray1}{\circle*{3}}}
\put(10,30){\textcolor{gray}{\circle*{3}}}
\put(20,30){\textcolor{blue}{\circle*{3}}}
\put(30,30){\circle*{3}}
\put(40,30){\circle*{3}}
\put(50,30){\circle*{3}}

% y=2 row
\put(0,40){\textcolor{gray1}{\circle*{3}}}
\put(10,40){\textcolor{blue}{\circle*{3}}}
\put(20,40){\textcolor{blue}{\circle*{3}}}
\put(30,40){\circle*{3}}
\put(40,40){\circle*{3}}
\put(50,40){\circle*{3}}

% y=3 row
\put(0,50){\textcolor{red}{\circle*{3}}}
\put(10,50){\textcolor{blue}{\circle*{3}}}
\put(20,50){\circle{3}}
\put(30,50){\circle*{3}}
\put(40,50){\circle*{3}}
\put(50,50){\circle*{3}}

% y=4 row
\put(0,60){\textcolor{gray1}{\circle*{3}}}
\put(10,60){\circle*{3}}
\put(20,60){\circle*{3}}
\put(30,60){\circle*{3}}
\put(40,60){\circle*{3}}
\put(50,60){\circle*{3}}

\end{picture}\\
\multicolumn{1}{c}{I=(x_1^4,\,x_2^4,\,x_1^2x_2)}
\end{array}
$$

Concretely, writing degrees in coordinates to match the diagram,
$\alpha=(1,0)$ \textup{(green)} has $m=2$,
$E_{2\alpha}=E^+_\alpha=\{e_{1}\}$, so Case~(ii.2) gives
$\iota_\alpha(N(\alpha))\cong\GA^2$; while $\alpha=(0,1)$
\textup{(gray)} has $m=2$, $E_{\alpha}=\{e_{1},e_{2}\}$,
$E^+_\alpha=\{e_{2}\}$, so neither case applies and
$\iota_\alpha(N(\alpha))\not\cong\GA^2$.
\end{example}

\subsection{The semidirect decomposition and the Levi subgroup}

Now, let $\M$ be the maximal ideal of $\KK[\XX]/I$. Since any automorphism $\varphi\in G$ stabilizes $\M$, we see that $G$ acts naturally (and linearly) on the quotient $\M^i/\M^{i+1}$ for every $i\geq 0$. Of course, this action is trivial for $i=0$, but one immediately checks for instance that the action of the torus $T\subset G$ on $\M/\M^2$ is faithful. This allows us to decompose $G$ into two parts.

\begin{proposition}\label{prop:semidirect}
With the notations from above, let $a:G\to\GL(\M/\M^2)$ be the morphism induced by the action of $G$ on $\M/\M^2$ and let $U_1:=\ker(a)$. For $\alpha\in M=\ZZ^n$, denote by $s(\alpha)$ the absolute degree of $\alpha$ (i.e. the sum of its coordinates). Then the following hold:
\begin{enumerate}[(i)]
    \item $G$ is isomorphic to the semidirect product $U_1\rtimes G_1$, where $G_1$ is the subgroup of $G$ generated by $T$, toric automorphisms, and the $U_{\alpha}$ with $s(\alpha)=0$.\label{item:semidirect}
    \item $U_1$ is a connected unipotent group generated by the $U_{\alpha,p}$ with $s(\alpha)>0$.\label{item:unipotent}
\end{enumerate}
\end{proposition}

\begin{remark}\label{remark:-e_i}
We have $s(\alpha)\geq 0$ for every root $\alpha$. Indeed, if $s(\alpha)<0$, then $\alpha$ is outer. Since outer roots must always have a unique negative coordinate, equal to $-1$, the only possibility is $\alpha=-e_i$. But $-e_i$ is never a root of an ideal with cofinite support: if $d\geq2$ is minimal with $x_i^d\in I$, then $\partial_{-e_i,\,e_i^*}$ sends $\XX^{d e_i}$ to $d\,\XX^{(d-1)e_i}\notin I$, so it does not preserve $I$. Consequently, every root of degree-sum zero is of the form $\alpha=e_{i}-e_{j}$ ($i\neq j$), for which the only nontrivial homogeneous derivations are $\overline\partial_{\alpha,\lambda e_j^*}$ with $\lambda\in\KK^*$, satisfying $\exp(\lambda \overline\partial_{\alpha,e_j^*})(\overline x_j)=\lambda \overline x_i+\overline x_j.$ Together with the linearity of the actions of $T$ and toric automorphisms on the variables, this shows that every element of $G_1$ is a ``linear automorphism'', in the sense that it acts linearly on the $\KK$-subspace of $\KK[\XX]/I$ spanned by the variables $\overline{x}_{i}$.
\end{remark}

\begin{proof}
Let us prove that $a$ is injective on $G_1$. As it was stated above, any element of $G_1$ corresponds to a $\KK$-linear substitution of the variables $\overline{x}_i$. Assume $g\in G_1\cap\ker(a).$ Then for every $i$ we have $g(\overline{x}_i)\equiv \overline{x}_i\pmod{\M^2}$. Since the variables $\overline{x}_i$ are $\KK$-linearly independent in $\M/\M^2$, this implies that $g(\overline{x}_i)=\overline{x}_i.$ Hence $g$ is trivial.

On the other hand, consider a root $\alpha\in M$ with $s(\alpha)>0$. Then the corresponding homogeneous derivations $\overline\partial_{\alpha,p}$ send the class of $\overline x_i$ in $\M/\M^2$ to that of a polynomial of higher degree, hence in $\M^2$. This implies that the corresponding $\GA$-actions are trivial on $\M/\M^2$ and thus the subgroups $U_{\alpha,p}\subset G$ are all contained in $U_1$.

Since there are no homogeneous derivations of $\KK[\XX]/I$ with negative absolute degree, we see by \cref{thm:dlmr} that $G$ is generated by $G_1$ and $U_1$. Injectivity of $a$ on $G_1$ tells us moreover that $G_1\cap U_1=\{1\}$. Since $U_1=\ker(a)$ is normal, we get \eqref{item:semidirect}.

In order to get \eqref{item:unipotent}, we prove first that the Lie algebra $\mathfrak u_1$ of $U_1$ is $\bigoplus_{s(\alpha)>0} \mathfrak g_\alpha$. Indeed, the Lie algebra of $G$ is, again by \cref{thm:dlmr} and \cref{remark:-e_i},
$$\mathfrak{g}=\bigoplus_{s(\alpha)\geq 0} \mathfrak g_\alpha=\left(\bigoplus_{s(\alpha)=0} \mathfrak g_\alpha\right)\oplus\left(\bigoplus_{s(\alpha)>0} \mathfrak g_\alpha\right).$$
But we also know that $\mathfrak{g}=\mathfrak g_1\oplus\mathfrak u_1$ because of the semidirect product structure, and clearly $\bigoplus_{s(\alpha)=0} \mathfrak g_\alpha\subseteq\mathfrak g_1$ and $\bigoplus_{s(\alpha)>0} \mathfrak g_\alpha\subseteq\mathfrak u_1$. This yields the desired equality.

Knowing this, in order to conclude it suffices to prove that $U_1$ is connected and unipotent. Now, the number of connected components of $G=U_1\rtimes G_1$ is precisely the product of those of $U_1$ and $G_1$ since a semidirect product of algebraic groups is just a product as a variety. But since $G_1$ contains all toric automorphisms and these generate the group of connected components by \cref{thm:dlmr}, $G_1$ has at least as many connected components as $G$, so that $U_1$ must be connected. Finally, $U_1$ is unipotent since it is connected and normal in $G$ and thus a maximal torus of $U_1$ is given by $T\cap U_1=\{1\}$. This concludes the proof.
\end{proof}

We keep the notations from above. Since $U_1$ is unipotent and normal in $G$, we see that $U_1$ is contained in the unipotent radical $R$ of $G$. In particular, if $G^0=R\rtimes L$ for some reductive group $L$, we see that $L$ is naturally a subgroup of $G/U_1$ via $a$ and thus $(G/U_1)^0=(R/U_1)\rtimes L$. The following result describes more precisely the group $L$.

\begin{proposition}\label{prop:levi}
The group $L$ is isomorphic to the subgroup $L_1$ of $G$ generated by $T$ and the $U_\alpha$ with $s(\alpha)=0$ and such that both $\alpha$ and $-\alpha$ are roots, i.e.~belong to $\mathcal R(I)$.

Moreover, if we use the basis $\{\overline x_i\}_{1\leq i\leq n}$ of $\M/\M^2$ to identify $\GL(\M/\M^2)$ with $\GL_n(\KK)$, then, up to permuting the variables, $L$ is a standard Levi subgroup of $\GL_n(\KK)$. In particular, it is isomorphic to $\prod_{j=1}^r\GL_{m_j}(\KK)$, where $m_1+\cdots+m_r=n$.
\end{proposition}

\begin{proof}
Since $G/U_1$ is isomorphic to $G_1$, we may work with this group instead. Moreover, this group can be seen as a subgroup of $\GL_n(\KK)$ via its action on $\M/\M^2$ with the basis fixed above. This identifies $T$ with the subgroup of diagonal matrices and hence we may identify the corresponding lattices of characters and cocharacters. Now, recall that for $\alpha$ with $s(\alpha)=0$ we have $\alpha=e_i-e_j$ for some $i\neq j$ and that an element in $U_\alpha$ sends the class of $\overline x_j$ in $\M/\M^2$ to that of $\lambda\overline x_i+\overline x_j$ for some $\lambda\in\KK$, while fixing the classes of $\overline x_k$ for $k\neq j,$ see \Cref{remark:-e_i}. This tells us that the root group $U_\alpha\subset G_1$ can be identified with the classical root group of $\GL_n(\KK)$ associated to the same root $\alpha=e_i-e_j$, which is given by matrices of the form $\mathrm{Id}+\lambda E_{ij}$ for some $\lambda\in\KK$. We will abusively denote this group by $U_\alpha$ as well. Thus, we may and will assume henceforth that $G_1$ sits inside $\GL_n(\KK)$ and that there exist a certain subset
$$\Sigma\subset\{e_i-e_j,\,i\neq j,\,1\leq i\leq n,\,1\leq j\leq n\}\subset M,$$
such that $G_1^0$ is generated by the diagonal torus $T$ and the root groups $U_\alpha$ of $\GL_n(\KK)$ for $\alpha\in \Sigma$. In particular, if we define $\Sigma_L$ as $\Sigma\cap-\Sigma$, then $L_1$ is the subgroup generated by $T$ and the $U_{\alpha}$ with $\alpha\in \Sigma_L$.

With all these assumptions, it is clear that $L_1$ is contained in $G_1^0$. Let $U_0$ be the unipotent radical of $G_1^0$. In order to prove the first assertion, it will suffice to prove that $G_1^0=U_0\rtimes L_1$. We will prove then first that $U_0$ and $L_1$ generate $G_1^0$, which we will do by proving that $U_\alpha$ lies in $U_0$ for every $\alpha\in \Sigma\smallsetminus \Sigma_L$.

Let us recall some basic facts about root groups in $\GL_n(\KK)$. As one can check by direct computations on matrices, if $\alpha=e_i-e_j$, then $U_\alpha$ is normalized by $T$ and commutes with $U_\beta$ for $\beta=e_k-e_\ell$ unless one of the following holds:
\begin{enumerate}[(i)]
    \item $k=j$ and $\ell=i$, i.e.~$\alpha=-\beta$. In this case $U_\alpha$ and $U_\beta$ generate, along with a suitable subgroup of the torus $T$, a copy of $\GL_2(\KK)$ that acts naturally on the $i$-th and $j$-th variables.
    \item $k=j$ and $\ell\neq i$. In this case, the commutator subgroup $[U_\alpha,U_\beta]$ is $U_\gamma$ with $\gamma=\alpha+\beta=e_i-e_\ell$.
    \item $k\neq j$ and $\ell=i$. In this case, the commutator subgroup $[U_\alpha,U_\beta]$ is $U_\gamma$ with $\gamma=\alpha+\beta=e_k-e_j$.
\end{enumerate}
Note that the last two cases are the only ones where $\alpha+\beta$ is a nontrivial root of $\GL_n(\KK)$. We will only use these two in what follows.

Assume that $\alpha\in \Sigma\smallsetminus \Sigma_L$ and $\beta\in\Sigma$. The computations above tell us then that either $U_\alpha$ and $U_\beta$ commute or they generate $U_{\alpha+\beta}$. We claim that in the latter case we have that $\alpha+\beta\in \Sigma\smallsetminus \Sigma_L$ as well. Indeed, since $U_\alpha,U_\beta\subset G_1^0$, we have $U_{\alpha+\beta}\subset G_1^0$ and hence $\alpha+\beta\in \Sigma$. On the other hand, if $\gamma=\alpha+\beta\in \Sigma_L$, then $-\gamma\in \Sigma$ by definition of $\Sigma_L$. And since $\beta+(-\gamma)=-\alpha$, we finally get via the same computations that $U_{-\alpha}\subset G_1^0$ and thus $-\alpha\in \Sigma$. So $\alpha\in \Sigma_L$ and we get a contradiction, proving our claim.

With this, \cite[Ch.~VI, \S1.7, Prop.~22]{BourbakiLie46} tells us that $\Sigma\smallsetminus \Sigma_L$ is contained in the set of positive roots of $\GL_n(\KK)$ for a suitable ordering; that is, up to permuting the variables, every $U_\alpha$ with $\alpha\in\Sigma\smallsetminus \Sigma_L$ consists of unipotent upper triangular matrices. Hence the subgroup $U':=\langle U_\alpha\mid\alpha\in\Sigma\smallsetminus \Sigma_L\rangle$ of $G_1^0$ is unipotent. Moreover, $U'$ is normalized by $T$ and, by the arguments above, by every $U_\beta$ with $\beta\in\Sigma$. Since $T$ and these $U_\beta$ generate $G_1^0$, $U'$ is normal in $G_1^0$, hence contained in its unipotent radical $U_0$. In particular $U_\alpha\subset U_0$ for every $\alpha\in\Sigma\smallsetminus\Sigma_L$, as we wanted.

Knowing that $L_1$ and $U_0$ generate $G_1^0$, we claim that $L_1$ is reductive. Assuming this, since $U_0$ is normal in $G_1^0$, we see that $U_0\cap L_1$ is a normal unipotent subgroup of $L_1$, hence of dimension 0. But there are no nontrivial finite unipotent groups since $\KK$ has characteristic 0. This proves that $G_1^0=U_0\rtimes L_1$ assuming the claim. Let us prove then that $L_1$ is a standard Levi subgroup, which clearly implies that it is reductive.

In order to prove this, note that if $\Sigma_L$ is empty, then $L_1=T$ and we are done. Otherwise, define an equivalence relation in $\Sigma_L/\{\pm 1\}$ by relating $\pm e_i\mp e_j$ with $\pm e_k \mp e_\ell$ if either $i=k$, $i=\ell$, $j=k$ or $j=\ell$ (and extending for transitivity). This immediately implies that $\pm \alpha$ is related with every $\pm\beta$ such that either $\alpha+\beta\in \Sigma$ or $\alpha-\beta\in \Sigma$, in which case we see that such a root is also in $\Sigma_L$ using the same argument as above. Consider now the class of $\pm\alpha=\pm e_i\mp e_j$ and the subset of $\{1,2,\ldots,n\}$ of indexes $k$ such that a $\pm\beta$ in the class has a nontrivial $k$-th coordinate. Up to permuting the variables, we may assume that $\alpha=\pm e_1 \mp e_2$ and that this set is $\{1,2\ldots,m\}$ for some $2\leq m\leq n$.

We claim then that the whole $\GL_m(\KK)$ is contained in $L_1$. Indeed, by definition, for every $3\leq j\leq m$ there is a sequence of elements in the class
$$\pm e_{i_1}\mp e_{i_2},\pm e_{i_2}\mp e_{i_3}, \ldots, \pm e_{i_{r-1}}\mp e_{i_r},$$
such that $i_1=1$ and $i_2=2$, and $i_r=j$. Then, by the same computations from above, we see that $U_{e_1-e_j}$ and $U_{e_j-e_1}$ are generated by the $U_{e_{i_j}-e_{i_{j+1}}}$ and the $U_{e_{i_{j+1}}-e_{i_j}}$, which are all in $L_1$ by definition. Thus $U_{e_1-e_j}$ and $U_{e_j-e_1}$ are both contained in $L_1$ for any $1\leq j\leq m$. And the same computation one more time gives us that $U_{e_i-e_j}$ and $U_{e_j-e_i}$ are contained in $L_1$ for any $1\leq i<j\leq m$, proving the claim.

We conclude the proof of the second statement by iterating this argument with the remaining coordinates. Since we started with an equivalence relation, the partition $n=m_1+m_2+\cdots+m_r$ it generates yields the desired product inside $L_1$. But this connected subgroup has the expected Lie algebra, so $L_1$ cannot be bigger than this product.
\end{proof}

Let us finish this section by proving that all of these Levi groups can appear in concrete cases. Given a standard Levi subgroup $L$ of $\GL_n(\KK)$, consider the corresponding standard parabolic subgroup $P_L=\langle L,U_n\rangle$, where $U_n\subseteq \GL_n(\KK)$ denotes the subgroup of unipotent upper triangular matrices.

\begin{proposition}\label{prop:levi exist}
There exists a monomial ideal $I$ such that the automorphism group $G$ of the corresponding monomial algebra satisfies $G_1^0\simeq L$. The same holds if we replace $L$ by $P_L$. In particular, in both cases the Levi subgroup of $G$ is $L$.
\end{proposition}

\begin{proof}
Consider the partition $n=m_1+m_2+\cdots+m_r$ such that $L=\prod_{k=1}^r\GL_{m_k}(\KK)$. Rename the variables $x_1,\ldots,x_n$ as $x_{i,j}$ with $1\leq i\leq r$ and $1\leq j\leq m_i$, ordered lexicographically. Choose integers $1<\ell_1<\ell_2<\cdots<\ell_r$ and, for $1\leq k\leq r$, consider the following sets of monomials:
$$B_{L,k}=\left\{\prod_{j=1}^{m_k}x_{k,j}^{\alpha_j}\mid \sum_{j=1}^{m_k}\alpha_j=\ell_k\right\},$$
that is, all the monomials of total degree $\ell_k$ in the variables $x_{k,j}$; and
$$B_{P_L,k}=\left\{\prod_{i=1}^{k}\prod_{j=1}^{m_i}x_{i,j}^{\alpha_{i,j}}\mid \sum_{i=1}^k\sum_{j=1}^{m_i}\alpha_{i,j}=\ell_k\right\},$$
that is, all the monomials of total degree $\ell_k$ in the variables $x_{i,j}$ with $i\leq k$. Define the monomial ideals $I_{L}=\langle B_{L,k}\mid 1\leq k\leq r\rangle$ and $I_{P_L}=\langle B_{P_L,k}\mid 1\leq k\leq r\rangle$. We claim that the automorphism groups $\Aut(\KK[\XX]/I_L)$ and $\Aut(\KK[\XX]/I_{P_L})$ have $L$ and $P_L$ as respective neutral connected components of their ``linear part''.

We already know that the Lie algebras of $L$, $P_L$, $\Aut(\KK[\XX]/I_L)$ and $\Aut(\KK[\XX]/I_{P_L})$ decompose as graded direct sums of root derivations (and the Lie algebra $\mathfrak{g}_0$ of the torus $T$), and that the corresponding degrees for the subgroup of ``linear automorphisms'' (i.e.~$G_1$) are those $\alpha$ with $s(\alpha)=0$. Moreover, since all such $\alpha$ are outer degrees, the corresponding root subspace $\mathfrak{g}_\alpha$ has dimension 1 and corresponds to the Lie algebra of an actual root subgroup. By the correspondence between subgroups of $\GL_n$ and Lie subalgebras of $\mathfrak{gl}_n$ \cite[Section~13.1, Theorem]{Humphreys}, it suffices then to check, for every such $\alpha$, that any derivation in $\mathfrak{g}_\alpha$ (equivalently, all of them) is well-defined on $\KK[\XX]/I_L$ (resp.~$\KK[\XX]/I_{P_L}$) if and only if $\alpha$ is a root of $L$ (resp.~$P_L$). 

Considering the renaming of the variables we did above, we note that the roots of $\GL_n(\KK)$ have the form $e_{i,j}-e_{k,l}$ for some $1\leq i,k\leq r$, $1\leq j\leq m_i$ and $1\leq l\leq m_k$, so that the roots of $L$ are precisely those of the form $e_{i,j}-e_{i,k}$, while those of $U_n$ are those of the form $e_{i,j}-e_{k,l}$ with $i\leq k$ and $j<l$ if $i=k$.

Let $\alpha=e_{i,j}-e_{k,l}$ be a root of $\GL_n(\KK)$. Then the derivation $\partial_{\alpha,e_{k,l}^*}$ sends $x_{k,l}$ to $x_{i,j}$ and every other variable to $0$. More precisely, on any given monomial, $\partial_{\alpha,e_{k,l}^*}$ replaces one appearance of $x_{k,l}$ by $x_{i,j}$ (up to scalar), and it gives $0$ if there are none. Let us go case by case then.
\begin{itemize}
    \item If $i=k$, then $\alpha$ is a root of $L$ and hence of $P_L$, so we must prove that $\partial_{\alpha,e_{k,l}^*}$ preserves both $I_L$ and $I_{P_L}$. But this is evident by definition of $B_{L,k}$ and $B_{P_L,k}$.
    \item If $i<k$, then $\alpha$ is not a root of $L$, but it is a root of $U_n$, hence of $P_L$. By definition of $B_{P_L,k}$, it is evident that $\partial_{\alpha,e_{k,l}^*}$ preserves $I_{P_L}$. On the other hand, any element in $B_{L,k}$ containing the variable $x_{k,l}$ is sent to an element whose total degree on the variables $x_{k,t}$ for $1\leq t\leq m_k$ is $\ell_{k}-1$, while its degree in the variables of any other block is $1<\ell_s$ (recall $\ell_s\geq2$), and hence is not in $I_L$.
    \item If $i>k$, then $\alpha$ is neither a root of $L$ nor of $P_L$. The argument is essentially the same as above: since $x_{i,j}\not\in B_{P_L,k}$, any element in $B_{L,k}\subseteq B_{P_L,k}$ is sent to an element whose total degree on the variables $x_{s,t}$ for $1\leq s\leq k$ and $1\leq t\leq m_s$ is $\ell_{k}-1$; hence its total degree on the variables of the blocks $s\leq s'$ is:\\
    \begin{center}
    \begin{tabular}{rll}
        $0$ & $<\ell_{s'}$ & for $s'<k$;\\
        $\ell_k-1$ & $<\ell_{s'}$ & for $k\leq s'<i$;\\
        $\ell_k$ & $<\ell_{s'}$ & for $s'\geq i$.
    \end{tabular}
    \end{center}
    Hence the image is divisible by no element of $B_{P_L,s'}$, and hence is not in $I_{P_L}\supseteq I_L$.
\end{itemize}

\end{proof}

\section{The component group and the realization of finite groups}\label{sec:fin-groups}
The component group $G/G^0$ is finite and generated by toric automorphisms, as established in \cref{thm:dlmr}. It is therefore natural to ask which finite groups occur as $G/G^0$. The aim of this section is to prove that every finite group does, using a construction based on Frucht's theorem on graph automorphisms.

By a graph we mean a finite simple undirected graph $H=(V,E)$: a finite
set $V$ of vertices together with a set $E\subseteq\binom{V}{2}$ of
two-element subsets of $V$, the edges; thus $H$ has neither loops nor
multiple edges. An automorphism of $H$ is a bijection $\sigma\colon V\to
V$ such that $\{i,j\}\in E$ if and only if $\{\sigma(i),\sigma(j)\}\in E$. Under
composition the automorphisms of $H$ form a group $\Aut(H)$, a subgroup of the
symmetric group $\Sym(V)$; when $V=\{1,\dots,n\}$ we view $\Aut(H)\subseteq S_n$.
Our construction rests on the following classical realization theorem.

\begin{theorem}[Frucht, {\cite[Sections~2--3]{Frucht1939}}]\label{thm:frucht}
For every finite group $\Gamma$ there exists a finite simple graph $H$ with
$\Aut(H)\cong\Gamma$.
\end{theorem}

Recall that toric automorphisms are the
automorphisms $\overline{x}_k\mapsto\overline{x}_{\sigma(k)}$ for $\sigma\in S_n$
preserving $I$, forming a subgroup $\Aut(S,I)\subseteq G$. Since by assumption $x_i\not\in I$, the classes $\overline{x}_1,\dots,\overline{x}_n$ are nonzero and pairwise distinct, so $\sigma$ is determined by the automorphism, and $\Aut(S,I)$ is isomorphic to
$$\Sym(I):=\{\sigma\in S_n\mid\sigma(I)=I\},$$
the group of variable permutations preserving the minimal monomial generating set of $I$.

Fix a graph $H$ on $V=\{1,\dots,n\}$ with edge set $E$ and an integer
$d\geq 3$. Let $I_H\subset\KK[\XX]$ be the ideal generated by the
edge monomials $x_ix_j$ with $\{i,j\}\in E$ together with the pure powers
$x_1^d,\dots,x_n^d$:
\begin{equation}\label{eq:IH}
I_H=\big(x_1^d,\dots,x_n^d,\;x_ix_j\mid\{i,j\}\in E\big).
\end{equation}
The edge monomials are squarefree of degree two and the powers
$x_k^d$ are supported on a single variable, so none divides another and
together they form the minimal monomial generating set of $I_H$. Since $x_k^d\in I_H$ for all $k$, $I_H$
has cofinite support.

\begin{lemma}\label{lem:frucht-ideal}
$\Aut(S,I_H)\cong\Aut(H)$.
\end{lemma}

\begin{proof}
A permutation $\sigma\in S_n$ lies in $\Sym(I_H)$ if and only if it preserves
the minimal generating set of $I_H$. Since the edge monomials have degree two
and the pure powers have degree $d\geq 3$, $\sigma$ must preserve each type
separately; it preserves the pure powers automatically, and preserves the edge
monomials if and only if it preserves $E$, that is, if and only if
$\sigma\in\Aut(H)$. Thus $\Sym(I_H)=\Aut(H)$ and the result follows from the
isomorphism $\Aut(S,I_H)\cong\Sym(I_H)$.
\end{proof}

Recall that, for $i\neq j$, since $e_i-e_j$ is outer, the unique (up to scalar) homogeneous derivation of $\KK[\XX]$ of degree $e_i-e_j$ is $\partial_{e_i-e_j,\,e_j^*}$.

\begin{lemma}\label{lem:no-bidir}
Let $d\geq3$. For every $i\neq j$ the derivation $\partial_{e_i-e_j,\,e_j^*}$ does not preserve $I_H$.
\end{lemma}

\begin{proof}
It suffices to exhibit a generator $\XX^a$ of
$I_H$ with $a_j\geq1$ and $\XX^{a-e_j+e_i}\notin I_H$. If $\{i,j\}\notin E$, take
$\XX^a=x_j^d$; then $\XX^{a-e_j+e_i}=x_j^{\,d-1}x_i$ has support $\{i,j\}$ and all
exponents $<d$, so it is divisible neither by any power $x_k^d$ nor by any edge
monomial (the only candidate, $x_ix_j$, is not a generator). If $\{i,j\}\in E$,
take $\XX^a=x_ix_j$; then $\XX^{a-e_j+e_i}=x_i^2$ has support $\{i\}$, so it is
divisible by no edge monomial and, since $2<d$, by no power. In both cases
$\XX^{a-e_j+e_i}\notin I_H$.
\end{proof}

\begin{theorem}\label{thm:realize-fin}
For every finite group $\Gamma$ there is a monomial ideal
$I\subset\KK[\XX]$ such that
$$
\Aut_\KK(\KK[\XX]/I)\big/\Aut^0_\KK(\KK[\XX]/I)\cong\Gamma.
$$
Moreover $I$ may be chosen so that
$\Aut_\KK(\KK[\XX]/I)\cong\Aut^0_\KK(\KK[\XX]/I)\rtimes\Gamma$ and
$\Aut^0_\KK(\KK[\XX]/I)$ is solvable.
\end{theorem}

\begin{proof}
By \cref{thm:frucht} choose a graph $H$ with $\Aut(H)\cong\Gamma$, and let
$I=I_H$ as in \eqref{eq:IH} with $d\geq3$. By \cref{lem:frucht-ideal},
$\Aut(S,I)\cong\Gamma$, and by \cref{thm:dlmr}\,(iii) the natural map
$\Aut(S,I)\to G/G^0$ is surjective. It remains to show that this map is also
injective, that is, that no nontrivial toric automorphism lies in $G^0$.

To detect whether a toric automorphism $\sigma$ belongs to $G^0$, we use the
morphism $a\colon G\to\GL(\M/\M^2)$ that sends each automorphism to its action
on the classes $\overline{x}_1,\dots,\overline{x}_n$ modulo $\M^2$. Let
$U_1 = \ker a$ be the subgroup of automorphisms acting trivially on $\M/\M^2$
(\cref{prop:semidirect}). Since $U_1$ is connected
(\cref{prop:semidirect}\,(\ref{item:unipotent})), it is contained in $G^0$, so
$\sigma\in G^0$ if and only if $a(\sigma)\in a(G^0)$. By
(\ref{item:semidirect}), $G^0 = U_1 \rtimes G_1^0$, so $a(G^0) = a(G_1^0)$.

We now determine $G_1^0$. Since $G_1^0$ is generated by $T$ and the root
subgroups $U_\alpha$ with $s(\alpha)=0$
(\cref{prop:semidirect}\,(\ref{item:semidirect})), it suffices to show there are
no such roots. By \cref{remark:-e_i}, a root of degree-sum zero is of the form $e_i-e_j$ with $i\neq j$, and the only homogeneous derivation of $\KK[\XX]$ of that degree is, up to scalar, $\partial_{e_i-e_j,\,e_j^*}$. By \cref{lem:no-bidir} none of these
preserves $I_H$, so $G_1^0 = T$ and $a(G^0) = a(T)$ is the diagonal torus of
$\GL(\M/\M^2)$.

On the other hand, a toric automorphism $\sigma$ permutes the classes
$\overline{x}_1,\dots,\overline{x}_n$, so $a(\sigma)$ is a permutation matrix;
such a matrix is diagonal only when $\sigma = \mathrm{id}$. Therefore
$\Aut(S,I)\cap G^0 = \{1\}$ and
$$
G/G^0\cong\Aut(S,I)\cong\Gamma.
$$

Since $\Aut(S,I)$ meets $G^0$ trivially and surjects onto $G/G^0$, it is a
complement, whence $G\cong G^0\rtimes\Aut(S,I)\cong G^0\rtimes\Gamma$.
Finally, since $G_1^0 = T$, the reductive part of $G^0$ equals the maximal
torus; therefore $G^0$ is solvable.
\end{proof}

\subsection{Examples}\label{subsec:families}
We illustrate the construction with two classical families of groups.

Symmetric groups. For the complete graph $H=K_n$ on $V=\{1,\dots,n\}$,
every permutation of the vertices preserves the edge set $\binom{V}{2}$, so
$\Aut(K_n)=S_n$. With $d\geq 3$ the ideal
$$
I_{K_n}=\big(x_1^d,\dots,x_n^d,\; x_ix_j\mid 1\le i<j\le n\big)
$$
satisfies $G/G^0\cong S_n$ by the construction from \cref{thm:realize-fin}.

Dihedral groups. For $n\geq 3$, let $H=C_n$ be the cycle on
$V=\{1,\dots,n\}$ with edges $\{k,k+1\}$ for $k=1,\dots,n-1$ and $\{n,1\}$.
An automorphism of $C_n$ must send a vertex to any of the $n$ vertices and its
neighbor to one of the two neighbors of the image, giving at most $2n$
automorphisms; the $n$ rotations and $n$ reflections realize this bound, so
$\Aut(C_n)\cong D_n$. With $d\geq 3$ the ideal
$$
I_{C_n}=\big(x_1^d,\dots,x_n^d,\; x_nx_1,\; x_kx_{k+1}\mid 1\le k\le n-1\big)
$$
satisfies $G/G^0\cong D_n$ by the construction from \cref{thm:realize-fin}.

\section{Generating sets for the automorphism group}\label{sec:generators}
In this section we look for small generating sets of $G^0$, and determine a minimal one in the generic case. Recall from \cref{thm:dlmr} that $G^0$ is generated by the torus $T$ and the one-parameter subgroups $U_{\alpha,p}$ for $\alpha\in\mathcal R(I)$ and $p\in N(\alpha)$. Not all of these are needed: in many cases the derivations of a given degree arise as Lie brackets of those in other degrees, making the corresponding subgroups redundant. To make this precise, we introduce the notion of irreducible degree and analyze when such reductions are well-founded. The irreducible degrees single out the root subgroups that must be retained regardless of any choices made, and we present an explicit algorithm returns a generating family whose set of degrees is minimal.

Throughout, $\mathfrak g=\Der(A)=\operatorname{Lie}(G)$ carries the $M$-grading
$\mathfrak g=\bigoplus_{\gamma}\mathfrak g_\gamma$, where $\gamma$ ranges over
$\{0\}\cup\mathcal R(I)$ and $\mathfrak g_0=\mathfrak t=\operatorname{Lie}(T)$.
Which root subgroups are needed is, degree by degree, the question of which root
derivations cannot be obtained as Lie brackets from those of ``lower'' degrees. The correct measure
of this is Lie-theoretic. For a degree $\gamma\neq 0$ put
$$
  B_\gamma \;=\; \sum_{\substack{\alpha+\beta=\gamma\\ \alpha,\beta\in\mathcal R(I)}}
  [\mathfrak g_\alpha,\mathfrak g_\beta]\ \subseteq\ \mathfrak g_\gamma,
  \qquad
  \mu(\gamma) \;=\; \dim_\KK\bigl(\mathfrak g_\gamma/B_\gamma\bigr).
$$

\begin{definition}\label{def:irreducible-degree}
A degree $\gamma\in\mathcal R(I)$ is reducible if $B_\gamma=\mathfrak
g_\gamma$, and irreducible otherwise. 

\end{definition}
By \cite[Section~13.1, Theorem]{Humphreys}, generating $G^0$ from $T$ and a family
of root subgroups amounts to generating $\mathfrak g$ from $\mathfrak t$ and the
corresponding root derivations. The quantity $\mu(\gamma)$ is therefore a
lower bound: every generating set for $\mathfrak g$ must include at
least $\mu(\gamma)$ root derivations in degree~$\gamma$. However,
$B_\gamma=\mathfrak g_\gamma$ does not by itself guarantee that the degree
$\gamma$ can be removed from the generating set. Indeed, $B_\gamma$ is defined
as a sum of brackets over all pairs of roots $\alpha,\beta$ with
$\alpha+\beta=\gamma$, and those summands may themselves be reducible.
Discarding every reducible degree at once is legitimate only when the resulting
chain of reductions is well-founded; circular dependencies can arise otherwise,
as the next example shows.

\begin{example}\label{ex:reducible-not-discardable}
Let $I=\M^2\subset\KK[x_1,x_2,x_3]$. Then
$A=\KK\oplus(\M/\M^2)$ and $G^0=\GL_3(\KK)$. The roots of
degree-sum zero are the six $e_i-e_j$ ($i\neq j$), each of which is reducible:
for instance, $e_1-e_3=(e_1-e_2)+(e_2-e_3)$ provides a nonzero bracket
$[\overline\partial_{e_1-e_2,\,e_2^*},\,
  \overline\partial_{e_2-e_3,\,e_3^*}]
= \overline\partial_{e_1-e_3,\,e_3^*}$,
so $B_{e_1-e_3}=\mathfrak g_{e_1-e_3}$. But these six roots recover each other
in a cycle: the decomposition of each one uses two others, and no root has a
decomposition into ``simpler'' summands. Discarding all six leaves only the
torus~$T$, which does not generate $G^0=\GL_3(\KK)$. The reducibility of every
$e_i-e_j$ is genuine (each $\mathfrak g_{e_i-e_j}$ is indeed spanned by
brackets) but the six reductions are mutually dependent and cannot all be
performed simultaneously.
\end{example}

\subsection{Reduction for inner roots} We now try to determine irreducible degrees and also ``discardable'' degrees (which is not the same as reducible). Throughout, we freely use the bracket formula
$[\overline\partial_{\alpha,p},\overline\partial_{\beta,q}]
=\overline\partial_{\alpha+\beta,\,r}$, with $r=p(\beta)q-q(\alpha)p$, of
\cref{rmk:lie-bracket}.

\begin{lemma}\label{lem:recovery-mixed}
Let $\gamma \in \ZZ^n_{\geq 0}\setminus\supp(I)$ be an inner root different
from $\ell e_i$ for all $i$. Then for every $r\in N_\KK$ there exist
$p, q \in N_\KK$ and inner roots $\alpha, \beta$ with $\alpha+\beta=\gamma$
such that
$[\overline\partial_{\alpha,p},\overline\partial_{\beta,q}]
=\overline\partial_{\gamma,r}$. In particular, any such root is reducible.
\end{lemma}

\begin{proof}
Without loss of generality $\gamma_n\neq 0$. Choose $\alpha=(\gamma_1,\dots,\gamma_{n-1},0)$,
$\beta=\gamma_n e_n$, $p=e_n^*$, and $q=(q_1,\dots,q_n)$ with
$$
q_i=\frac{r_i}{\gamma_n}\quad(1\le i\le n-1),\qquad
q_n=\frac{1}{\gamma_n}\Bigl(r_n+\frac{1}{\gamma_n}
\sum_{i=1}^{n-1}r_i\gamma_i\Bigr).
$$
Both $\alpha$ and $\beta$ are inner roots: they are nonzero (since
$\gamma_n\neq 0$ and $\gamma\neq\ell e_n$) and dominated componentwise by
$\gamma\notin\supp(I)$. Moreover $\alpha+e_n\leq\gamma$ and, choosing $i<n$
with $\gamma_i\neq0$, also $\beta+e_i\leq\gamma$, so that
$e_n\in E_\alpha$ and $e_i\in E_\beta$ and neither $E_\alpha$ nor $E_\beta$
is empty. Since $p(\beta)=\gamma_n$ and
$q(\alpha)=\frac{1}{\gamma_n}\sum_{i=1}^{n-1}r_i\gamma_i$, the covector
$p(\beta)q-q(\alpha)p$ equals $r$. By \cref{rmk:lie-bracket},
$[\overline\partial_{\alpha,p},\overline\partial_{\beta,q}]
=\overline\partial_{\gamma,r}$.
\end{proof}

\begin{lemma}\label{lem:recovery-multiple}
Let $\gamma$ be an inner root. If $\gamma = \ell e_i$ with $\ell > 2$, then for every
$r \in N_\KK$ there exist $p, q \in N_\KK$ and inner roots
$\alpha, \beta$ different from $\gamma$ with $\alpha + \beta = \gamma$
such that
$[\overline\partial_{\alpha,p}, \overline\partial_{\beta,q}]
= \overline\partial_{\gamma,r}$. In particular, any such root is reducible.
\end{lemma}

\begin{proof}
Without loss of generality $i = n$. Set $\alpha = (\ell-1)e_n$ and
$\beta = e_n$; both are inner roots different from $\gamma$ since
$\ell > 2$ (indeed $\alpha+e_n=\gamma\notin\supp(I)$ and
$\beta+e_n=2e_n\leq\gamma$, so $e_n$ lies in $E_\alpha$ and in $E_\beta$). If $r_n = 0$, take $p = (0,\dots,0,1)$ and
$q = (r_1,\dots,r_{n-1},0)$; then $p(\beta) = 1$, $q(\alpha) = 0$,
and $p(\beta)q - q(\alpha)p = q = r$.
If $r_n \neq 0$, take $q = (r_1,\dots,r_n)$ and
$p = (p_1,\dots,p_n)$ with
$$
p_n = \frac{1}{2-\ell}, \qquad
p_i = \frac{r_i}{r_n(2-\ell)} \quad (i < n).
$$
Then $p(\beta) = p_n$ and $q(\alpha) = (\ell-1)r_n$, so the $n$-th
component of $p(\beta)q - q(\alpha)p$ is
$p_n r_n - (\ell-1)r_n p_n = r_n p_n(2-\ell) = r_n$, and for $i < n$
it is $p_n r_i - (\ell-1)r_n p_i = r_i$. By \cref{rmk:lie-bracket},
$[\overline\partial_{\alpha,p}, \overline\partial_{\beta,q}]
= \overline\partial_{\gamma,r}$.
\end{proof}

\begin{lemma}\label{prop:two-ei-criterion}
Let $2e_i$ be a root. Then $2e_i$ is reducible if and only if one of the
following holds:
\begin{enumerate}[(i)]
\item $3e_i\in\supp(I)$;
\item there is $k\neq i$ such that $2e_i-e_k$ is a root and
      $e_i+e_k\notin\supp(I)$;
\item there is $k\neq i$ such that $e_k-e_i$ and $3e_i-e_k$ are roots.
\end{enumerate}
Moreover, if it is irreducible (i.e.~if none of these conditions holds), then $e_i\in E_{2e_i}$ and $B_{2e_i}=\langle\overline\partial_{2e_i,e_j^*}\mid e_j\in E_{2e_i},\ j\neq i\rangle$, which is a hyperplane of $\mathfrak g_{2e_i}$; hence $\mu(2e_i)=1$.
\end{lemma}

\begin{proof}
By \cref{rmk:lie-bracket},
$[\overline\partial_{e_i,p},\overline\partial_{e_i,q}]
=\overline\partial_{2e_i,r}$ with $r=p(e_i)q-q(e_i)p$, whose $e_i^*$-component is $p(e_i)q(e_i)-q(e_i)p(e_i)=0$; as $p,q$ range over $\mathfrak g_{e_i}$ these brackets exhaust the subspace $\{\overline\partial_{2e_i,p}:p(e_i)=0\}$ of $\mathfrak g_{2e_i}$, which is therefore contained in $B_{2e_i}$. This immediately yields the last assertion.

If $3e_i\in\supp(I)$, then $e_i\notin E_{2e_i}$ and hence $\mathfrak g_{2e_i}$ is spanned by elements $\overline\partial_{2e_i,p}$ with $p(e_i)=0$; hence $B_{2e_i}=\mathfrak g_{2e_i}$ and $2e_i$ is reducible. This
is case~(i).

Assume now $3e_i\notin\supp(I)$, so $\mathfrak g_{2e_i}$ has one further
direction $\overline\partial_{2e_i,e_i^*}$; then $2e_i$ is
reducible if and only if this direction lies in $B_{2e_i}$, necessarily
contributed by a decomposition $2e_i=\alpha+\beta$ with an outer summand.\\

Suppose $\alpha$ is inner and $\beta$ is outer, say in $\R_k(I)$. Since $-e_i$ is not a root, $\beta_j>0$ for at least one $j\neq i$, hence $\alpha_j+\beta_j>0$ and thus the case $k=i$ is impossible. So $k\neq i$, and matching coordinates forces $\alpha=\alpha_ie_i+e_k$ and
$\beta=(2-\alpha_i)e_i-e_k$ with $\alpha_i\in\{0,1,2\}$. Again, $\alpha_i=2$ is excluded because $\beta=-e_k$ is never a root. So there remain two possibilities: $\alpha=e_k$ with $\beta=2e_i-e_k$, or $\alpha=e_i+e_k$ with $\beta=e_i-e_k$.

In the first, writing $q=e_k^*$ and $p$ for the covector of $\alpha$, \cref{rmk:lie-bracket} gives $r(e_i)=p(\beta)q_i-q(\alpha)p_i=-p_i$. This is nonzero for a suitable $p\in N(e_k)$ precisely when $e_i+e_k\notin\supp(I)$, in which case $e_k$ is indeed a root since $e_i\in E_{e_k}$. This is case~(ii).

In the second we get $r(e_i)=-p_i$, which is nonzero for a suitable $p\in N(e_i+e_k)$ precisely when $2e_i+e_k\notin\supp(I)$. But then
$e_i+e_k\notin\supp(I)$ as well, so $e_k$ is a root again. Moreover, when $e_i-e_k$ is an outer root and $2e_i\notin\supp(I)$ (which is the case here), $2e_i-e_k$ is a root as well, so that case~(ii) already holds for the same $k$. This decomposition therefore contributes nothing beyond the first.\\

Suppose instead both summands are outer, lying in $\R_j(I)$ and $\R_k(I)$ respectively;
we have $j\neq k$, since otherwise the corresponding coordinate of $2e_i$
would be $-2$. With $p=e_j^*$ and $q=e_k^*$ we get $r(e_i)=p(\beta)\,q_i-q(\alpha)\,p_i=\beta_j\,\delta_{ki}-\alpha_k\,\delta_{ji}$, which vanishes unless $j=i$ or $k=i$. Thus only decompositions with one
summand in $\R_i(I)$ contribute. Matching coordinates,
any such decomposition is of the form $(e_k-e_i)+(3e_i-e_k)$ for some $k\neq i$, which is case~(iii).
\end{proof}

\begin{example}\label{ex:two-ei}
Let $I=(x_1^4,\,x_1^3x_2,\,x_2^3)\subset\KK[x_1,x_2]$. Both $2e_1$ and $2e_2$ are roots, and
\cref{prop:two-ei-criterion} distinguishes them. For $2e_2$ one has
$3e_2\in\supp(I)$, so $2e_2$ is reducible by case~(i): its whole space is
recovered from $[\mathfrak g_{e_2},\mathfrak g_{e_2}]$. For $2e_1$, instead,
$3e_1\notin\supp(I)$, and neither case~(ii) nor case~(iii) applies, since
neither $2e_1-e_2$ nor $e_2-e_1$ are roots. Hence $2e_1$ is
irreducible, with $\mu(2e_1)=1$: the derivation $\overline\partial_{2e_1,\,e_1^*}$ is not a bracket of roots of ``lower'' degree.
\end{example}

For the degrees $e_i$ the situation is a bit more subtle. In order to study $B_{e_i}$, we need to record then all the possible decompositions.

\begin{lemma}\label{lem:ei-decompositions}
Let $e_i$ be a root and let $e_i=\alpha+\beta$ with
$\alpha,\beta\in\mathcal R(I)$. Then, after possibly exchanging the two
summands, exactly one of the following holds:
$$
\begin{array}{lll}
\textup{(A)} & \alpha=e_i-e_k, & \beta=e_k,\\[2pt]
\textup{(B)} & \alpha=e_k-e_i, & \beta=2e_i-e_k,\\[2pt]
\textup{(C)} & \alpha=e_k-e_l, & \beta=e_i-e_k+e_l,
\end{array}
$$
with $k,l,i$ pairwise distinct.
\end{lemma}

\begin{proof}
Since $s(e_i)=1$ and inner roots have $s\geq1$, at most one summand is inner. If $\beta$ is inner, then $s(\beta)=1$ and $s(\alpha)= 0$. Thus, $\beta=e_k$ for some $k$ and then $\alpha=e_i-e_k$
with $k\neq i$ (as $\alpha\neq0$). This is (A).

If both are outer, say in directions $k$ and $l$, then $k\neq l$, since otherwise that coordinate of
$e_i$ would equal $-2$. Comparing coordinates, the summand outer in direction $k$ has $k$-th coordinate $-1$ and the other has $(e_i)_k+1$, while the remaining coordinates vanish, except for the one carrying $e_i$.
This gives (B) when $i\in\{k,l\}$ and (C) otherwise.
\end{proof}

These three types of decomposition have different consequences on the subspace $B_{e_i}$ we are trying to understand.

\begin{lemma}\label{prop:ei-Iouter-search}
Let $e_i$ be a root and consider the subspace $D_{e_i}:=\{\overline\partial_{e_i,p}\mid p(e_i+\mathbf{1})=0\}\subseteq\mathfrak g_{e_i}$, where $\mathbf{1}=e_1+\cdots+e_n$. If there is a decomposition of type \textup{(A)}, then $B_{e_i}=[\mathfrak g_{e_i-e_k},\mathfrak g_{e_k}]=\mathfrak g_{e_i}$ and $[\mathfrak g_{e_i-e_k},D_{e_k}]=D_{e_i}$. Otherwise, $B_{e_i}$ is the sum of the following subspaces, one for each decomposition of \cref{lem:ei-decompositions}:
$$
\begin{array}{ll}
\textup{(B)} & \bigl\langle \overline{\partial}_{e_i,2e_k^*-e_i^*}\bigr\rangle\subseteq D_{e_i},\\[4pt]
\textup{(C)} & \bigl\langle \overline{\partial}_{e_i,e_k^*-e_l^*}\bigr\rangle\subseteq D_{e_i}.
\end{array}
$$
In particular $e_i$ is irreducible if and only if there are no roots of the form $e_i-e_k$.
\end{lemma}

\begin{proof}
By \cref{lem:ei-decompositions} every decomposition of $e_i$ is of type (A), (B)
or (C), so $B_{e_i}$ is the sum of the corresponding brackets. In each case the
covector of the bracket is $r=p(\beta)q-q(\alpha)p$ by \cref{rmk:lie-bracket}, and
the covector of an outer root is a multiple of $e_j^*$, where $e_j^*$ is its
direction.

For the decomposition (A), $\alpha=e_i-e_k$ is outer in the direction $e_k^*$, so we may set $p=e_k^*$, while $\beta=e_k$ is inner and hence $q\in N_\KK$ is arbitrary (even though it might give a trivial derivation, this is valid). Since $p(\beta)=1$ and
$q(\alpha)=q_i-q_k$, we get $r=q-(q_i-q_k)e_k^*$. Taking $q=e_j^*$ gives:
\[r=e_j^*\,\text{ for }\,j\neq i,k,\quad r=2e_k^*\,\text{ for }\,j=k,\quad r=e_i^*-e_k^*\,\text{ for }\,j=i.\]
This clearly spans the whole $N_\KK$, so we get all possible derivations in $\mathfrak g_{e_i}$. If we assume now that $q\in D_{e_k}$, we know that $q_k+\sum_{j=1}^n q_j=0$, and hence
\[r_i+\sum_{j=1}^n r_j=2r_i+r_k+\sum_{j\neq i,k} r_j=2q_i+(2q_k-q_i)+\sum_{j\neq i,k} q_j=q_k+\sum_{j=1}^n q_j=0,\]
so that $r\in D_{e_i}$. The fact that the image is the whole $D_{e_i}$ is a dimension count.

For the decomposition (B), $\alpha=e_k-e_i$ is outer in the direction $e_i^*$ and
$\beta=2e_i-e_k$ is outer in the direction $e_k^*$, so $p=e_i^*$ and $q=e_k^*$.
Then $p(\beta)=2$ and $q(\alpha)=1$, whence $r=2e_k^*-e_i^*$.

For the decomposition (C), $\alpha=e_k-e_l$ is outer in the direction $e_l^*$ and
$\beta=e_i-e_k+e_l$ is outer in the direction $e_k^*$, so $p=e_l^*$ and
$q=e_k^*$. Then $p(\beta)=1$ and $q(\alpha)=1$, whence $r=e_k^*-e_l^*$.

Summing over all admissible $k,l$ gives the stated description of $B_{e_i}$. Finally, $e_i$ is irreducible precisely when this sum is a proper subspace of $\mathfrak g_{e_i}$. This is the case when there are no decompositions of type \textup{(A)} since, as $\mathfrak g_{e_i}$ is generated by the $e_j^*$ with $e_j\in E_{e_i}$ and $\overline\partial_{e_i,e_j^*}\not\in D_{e_i}$, we have $D_{e_i}\neq\mathfrak g_{e_i}$ unless both are trivial (and in this case $e_i$ is not a root).
\end{proof}

\begin{example}\label{ex:ei-partial}
Let $I=(x_1^2,\,x_2^2,\,x_3^2,\,x_1x_3)\subset\KK[x_1,x_2,x_3]$. The three degrees
$e_1,e_2,e_3$ are roots, as well as $e_1-e_3$ and $e_3-e_1$. For $e_1$ one has
$\dim\mathfrak g_{e_1}=1$, and the type-(A) decomposition $e_1=(e_1-e_3)+e_3$
recovers its single direction, so $e_1$ is reducible; by symmetry so is $e_3$.
For $e_2$, however, $\dim\mathfrak g_{e_2}=2$, with basis $\{\overline{\partial}_{e_2,e_1^*},\overline{\partial}_{e_2,e_3^*}\}$. The
only decompositions of $e_2$ are of type (C), with $\alpha=e_3-e_1$ or
$\alpha=e_1-e_3$, and each contributes the covector $e_1^*-e_3^*$ to $B_{e_2}$.
Hence $B_{e_2}$ is one-dimensional and $\mu(e_2)=1$: the degree $e_2$ is
irreducible, but only one of its two directions must be retained, the other
being recovered as a bracket.
\end{example}

\cref{lem:recovery-multiple,lem:recovery-mixed,prop:two-ei-criterion} can be used to immediately reduce the set of generators by discarding those that are reducible. However, as \cref{ex:reducible-not-discardable} shows, there are some circularity issues generated by the roots of the form $e_i-e_j$, and these actually affect the roots of the form $e_i$. So one cannot do the same reduction with \cref{prop:ei-Iouter-search}. In any case, there is an intelligent way of discarding roots that yields the following result.

Define on $E$ the relation $e_i\succ e_j$ if $e_i-e_j\in\R(I)$ and fix a set $E_\mathrm{min}\subseteq E\cap\R(I)$ of ``minimal roots'', i.e.~such that for every $e_k\in E \cap\R(I)$ there exists $e_i\in E_\mathrm{min}$ with $e_k\succ e_i$ or $e_k=e_i$. Such a set exists because $\succ$ is transitive: if $e_i-e_j$ and $e_j-e_k$ are roots with $i\neq k$, take $\mathbf a\in\supp(I)$ with $a_k>0$; then $\mathbf a-e_k+e_j\in\supp(I)$ since $\partial_{e_j-e_k,e_k^*}$ preserves $I$, and its $j$-th coordinate is positive, so $\mathbf a-e_k+e_i\in\supp(I)$ since $\partial_{e_i-e_j,e_j^*}$ preserves $I$. Hence $\partial_{e_i-e_k,e_k^*}$ preserves $I$, and the induced derivation is nonzero because $(e_i-e_k)+e_k=e_i\notin\supp(I)$; that is, $e_i\succ e_k$.

Given the transitivity of $\succ$, we may and will assume henceforth that $E_\mathrm{min}$ satisfies moreover that there is no relation $e_j\succ e_i$ with $e_i,e_j\in E_{\mathrm{min}}$ distinct. This is done simply by succesively searching for such pairs within the remaining set and dropping one of the elements.

\begin{proposition}\label{lem:discardable}
The Lie algebra $\mathfrak g$ is generated by $\mathfrak t\cup\bigcup_{\alpha}\mathfrak g_\alpha$, where $\alpha$ ranges over all outer roots, all inner roots $e_i$ with $e_i\in E_\mathrm{min}$ and all irreducible roots $2e_i$.
\end{proposition}

\begin{proof}
Let $\gamma\in\mathcal R(I)$ be a root. It suffices to prove that $\mathfrak g_\gamma$ is generated by this set. If $\gamma$ is outer, the result is obvious. If $\gamma=e_k$, then either $e_k\in E_\mathrm{min}$ and the result is obvious or there exists $e_i\in E_\mathrm{min}$ with $e_k\succ e_i$. This implies that $e_k-e_i\in\R(I)$ and hence we are done by \cref{prop:ei-Iouter-search}. If $\gamma=2e_i$ and is irreducible, then the result is also obvious. We may assume then that $\gamma$ is inner, $\gamma\neq e_i$ and $\gamma\neq 2e_i$ unless one of the hypotheses of \cref{prop:two-ei-criterion} is met (i.e.~unless $2e_i$ is reducible).

We argue by induction on $s(\gamma)\geq2$. If $\gamma$ has at least two nonzero
coordinates, \cref{lem:recovery-mixed} writes every element of
$\mathfrak g_\gamma$ as a Lie bracket
$[\overline\partial_{\alpha,p},\,\overline\partial_{\beta,q}]$ with
$\alpha,\beta$ inner and
$s(\alpha),s(\beta)<s(\gamma)$. If $\gamma=\ell e_i$ with $\ell\geq3$,
\cref{lem:recovery-multiple} decomposes it as $(\ell-1)e_i+e_i$, again with
both summands inner and of strictly smaller degree sum. Finally, if $\gamma=2e_i$ and one of the hypotheses of \cref{prop:two-ei-criterion} is met, then every element of $\mathfrak g_{2e_i}$ can be written as a sum of Lie brackets of pairs in either $\mathfrak g_{e_i}^2$, $\mathfrak g_{2e_i-e_k}\times\mathfrak g_{e_k}$ or $\mathfrak g_{e_k-e_i}\times\mathfrak g_{3e_i-e_k}$.

In every case, the summands are either outer, in $\mathfrak g_{e_i}$ for some $i$, in $\mathfrak g_{2e_i}$ for some irreducible $2e_i$, or satisfy the inductive hypothesis, so $\mathfrak g_\gamma$ is generated as claimed.
\end{proof}

\cref{lem:discardable} tells us that a minimal generating set for $\mathfrak g$ (and hence for $G^0$) involves only the degrees in three families, namely the outer roots, some ``minimal'' inner roots $e_i\in E_\mathrm{min}$ and the irreducible roots $2e_i$, but the
containment is in general strict for outer roots. We now try to determine which outer roots can be discarded, reading the criteria off the minimal generators of $I$.

\subsection{Reduction for outer roots} For the outer family one can actually detect irreducible degrees and discard some other outer degrees in an effective way: the starting point is the fact that the outer roots in each $\R_i(I)$ are computed from the generators of $I$ by iterated colon ideals.

Recall that for two ideals $I,J\subseteq\KK[\XX]$, the colon ideal
(or ideal quotient) is $I:J=\{f\in\KK[\XX] : fJ\subseteq I\}$.
By \cite[Theorem~2.2.1]{B95}, the derivation
$\partial_{\beta-e_i,\,e_i^*}$ of $\KK[\XX]$ preserves $I$ if and only if
$\XX^\beta\in J_i:=I:\bigl(I:(x_i)\bigr)$.

\begin{remark}\label{rmk:outer-colon}
The outer roots in $\R_i(I)$ are thus the degrees $\alpha=\beta-e_i$ with $\beta_i=0$,
$\XX^\beta\in J_i$ and $\XX^\beta\notin I$. Since $J_i$ is a monomial ideal, it is determined by its minimal
generating set $G(J_i)$. Put
$$
\widetilde G_i=\bigl\{\XX^\beta\in G(J_i)\ :\
\beta_i=0,\ \XX^\beta\notin I\bigr\},
\qquad
\widetilde G(I)=\bigcup_{i=1}^n\bigl\{\beta-e_i\ :\
\XX^\beta\in\widetilde G_i\bigr\},
$$
so that $\widetilde G_i$ consists of monomials and $\widetilde G(I)$ of the
corresponding outer degrees. Both are finite and are obtained directly from the minimal generators of $I$ by an algorithmic procedure (see \cite[Proposition~1.2.2]{HH11} for the algorithm yielding $G(J_i)$). Note also that $\tilde G(I)$ contains every root of the form $e_j-e_i$ since there is no way that $\XX^{e_j}\in J_i$ and is not in a minimal generating set.
\end{remark}

\begin{lemma}\label{lem:outer-minimal}
Let $\gamma\in\R_i(I)$. Then
$\gamma\in\widetilde G(I)$ if and only if there is no decomposition
$\gamma=\alpha+\delta$ with $\alpha\in\R_i(I)$ and
$\delta\in\ZZ^n_{\geq 0}$ an inner root.
\end{lemma}

\begin{proof}
Suppose first that $\gamma\notin\widetilde G(I)$, that is,
$\XX^{\gamma+e_i}$ is not a minimal generator of $J_i$. Then some
proper divisor $\XX^{\gamma+e_i-\delta}$, with
$\delta\in\ZZ^n_{\geq 0}$ nonzero, lies in $J_i$. This divisor is not in
$I$, since otherwise $\XX^{\gamma+e_i}$ would be in $I$
as well. Its $i$-th exponent equals $-\delta_i$, hence $\delta_i=0$ and
$(\gamma-\delta)_i=-1$. Therefore $\gamma-\delta$ is an outer root in the direction $e_i^*$, and $\gamma=(\gamma-\delta)+\delta$ is a decomposition of the stated form. Finally, since $1\notin J_i$, we have $\gamma-\delta\neq -e_i$, so there is $j\neq i$ with $(\gamma-\delta)_j>0$. For such $j$, we have $e_j\in E_{\delta}$ since otherwise $\XX^{\gamma+e_i}=\XX^{\gamma+e_i-e_j-\delta}\XX^{\delta+e_j}$ would be in $I$, hence $\delta$ is an inner root.

Conversely, suppose $\gamma=\alpha+\delta$ with $\alpha\in\R_i(I)$ and $\delta\in\ZZ^n_{\geq 0}$ nonzero. Then
$\XX^{\alpha+e_i}\in J_i$ and $\XX^{\alpha+e_i}$ properly
divides $\XX^{\gamma+e_i}$, so the latter is not a minimal
generator of $J_i$ and $\gamma\notin\widetilde G(I)$.
\end{proof}

\begin{example}\label{ex:outer-colon}
Let $I=(x_1^2,\,x_1x_2,\,x_2^3)\subset\KK[x_1,x_2]$. For $i=1$,
$$
I:(x_1)=(x_1,x_2),\qquad J_1=I:(x_1,x_2)=(x_1,\,x_2^2),
$$
so $\widetilde G_1=\{x_2^2\}$ and $\mathcal R_1(I)=\{-e_1+2e_2\}$, with
derivation $\overline\partial_{-e_1+2e_2,\,e_1^*}$. For
$i=2$,
$$
I:(x_2)=(x_1,x_2^2),\qquad J_2=I:(x_1,x_2^2)=(x_1,x_2),
$$
so $\widetilde G_2=\{x_1\}$ and $\mathcal R_2(I)=\{e_1-e_2\}$, with
derivation $\overline\partial_{e_1-e_2,\,e_2^*}$. Here every
element of $\widetilde G(I)$ is an irreducible degree.
\end{example}

\begin{example}\label{ex:m2-outer-outer}
The elements of $\widetilde G(I)$ need not be irreducible, since
\cref{lem:outer-minimal} only rules out decompositions with a summand outer
in the same direction. Revisiting \cref{ex:reducible-not-discardable}, let $I=\M^2=(x_1^2,x_2^2,x_3^2,x_1x_2,x_1x_3,x_2x_3)
\subset\KK[x_1,x_2,x_3]$. Then $I:(x_1)=\M$ and
$J_1=I:\M=\M$, so
$$
\widetilde G_1=\{x_2,\,x_3\}\quad \text{and}\quad 
\mathcal R_1(I)=\{-e_1+e_2,\ -e_1+e_3\}\subseteq \widetilde G(I).
$$
Yet we know from \cref{ex:reducible-not-discardable} that every outer root of the form $e_i-e_j$ is reducible.
\end{example}

\begin{lemma}\label{lem:outer-reducible}
Let $\gamma\in\R_i(I)$ be an outer root. If $\gamma = \alpha + \beta$ with
$\alpha, \beta \in \mathcal R(I)$, then $[\mathfrak g_{\alpha},\mathfrak g_\beta]=\mathfrak g_{\gamma}$. In particular, $\gamma$ is reducible.

If moreover $\gamma\in\widetilde G(I)$ then
\begin{enumerate}[(i)]
\item both $\alpha$ and $\beta$ are outer;
\item their directions are distinct, and one of them is $e_i^*$;
\item the decomposition can be chosen so that the summand in $\R_i(I)$ lies in $\widetilde G(I)$.
\end{enumerate}
\end{lemma}

\begin{proof}
Since $\dim\mathfrak g_\gamma=1$, it suffices to exhibit one nonzero bracket in
$[\mathfrak g_\alpha,\mathfrak g_\beta]\subseteq\mathfrak g_\gamma$ to prove the first assertion. Note that the summands cannot both be inner: two elements of
$\ZZ^n_{\geq 0}$ add up to a vector with nonnegative coordinates, whereas
$\gamma_i=-1$. So either exactly one of them is inner or both are outer.

Suppose first that one summand is inner, say $\alpha$. Comparing the $i$-th
coordinates, we must have $\beta\in\R_i(I)$ and $\alpha_i=0$. Then taking
$p=q=e_i^*$ in \cref{rmk:lie-bracket} gives
$r=e_i^*(\beta)e_i^*-e_i^*(\alpha)e_i^*=(\beta_i-\alpha_i)\,e_i^*=-e_i^*$. Since $\overline\partial_{\gamma,e_i^*}$ generates $\mathfrak g_\gamma$, we are done. (Note that this is valid even if a priori $\overline\partial_{\alpha,e_i^*}$ might be trivial; a posteriori, it is not.)

Suppose now that both summands are outer, in the directions $e_j^*$ and $e_k^*$ respectively.  Comparing coordinates, we see that $j\neq k$ and $i\in\{j,k\}$. Assume
without loss of generality that $k=i$. Then $\alpha_i=\gamma_i-\beta_i=0$ and
$\beta_j=\gamma_j-\alpha_j=\gamma_j+1\geq1$. Taking $p=e_j^*$ and $q=e_i^*$ in
\cref{rmk:lie-bracket} gives $r=e_j^*(\beta)e_i^*-e_i^*(\alpha)e_j^*=\beta_j\,e_i^*$, whose corresponding derivation generates $\mathfrak g_\gamma$ because $\beta_j\neq0$.\\

Assume now that $\gamma\in\widetilde G(I)$. Then (i) is a direct consequence of \cref{lem:outer-minimal} and (ii) was already proved above. Let us prove (iii). By (ii) we may write $\gamma=\alpha+\beta$ with $\alpha\in\R_i(I)$ and $\beta\in\R_k(I)$ with $k\neq i$. Suppose $\alpha\notin\widetilde G(I)$. By
\cref{lem:outer-minimal}, $\alpha=\alpha'+\delta$ with $\alpha'\in\R_i(I)$ and $\delta\in\ZZ^n_{\geq 0}$ nonzero. If $\delta_k>0$, then $\delta+\beta\in\ZZ^n_{\geq 0}$ is nonzero and $\gamma=\alpha'+(\delta+\beta)$ contradicts \cref{lem:outer-minimal}. So $\delta_k=0$ and hence $\beta':=\delta+\beta$ satisfies $\beta'_k=-1$ and $\beta'_m\geq 0$ for $m\neq k$.

We check that $\beta'\in\R_k(I)$. It preserves the ideal since so does $\beta$ and $\mathbf{a}+\beta\in\supp(I)$ implies $\mathbf{a}+\beta+\delta\in\supp(I)$. To check $E_{\beta'}\neq\varnothing$, suppose $\beta'+e_k\in\supp(I)$. As $k\neq i$ and $\gamma\in\R_i(I)$ we have $\gamma_k\geq 0$, and
$\gamma_k=\alpha'_k+\delta_k+\beta_k=\alpha'_k-1$, so $\alpha'_k\geq1$.
Then $\alpha'+e_i-e_k\in\ZZ^n_{\geq 0}$ and
$$
\gamma+e_i=(\alpha'+e_i-e_k)+(\beta'+e_k),
$$
which would force $\gamma+e_i\in\supp(I)$, contradicting that $\gamma\in\R_i(I)$. Hence $\gamma=\alpha'+\beta'$ is again a decomposition of the same form. Since $\delta\neq0$, iterating strictly decreases the summand in the componentwise order, so the process terminates with that summand in $\widetilde G(I)$.
\end{proof}

\begin{remark}
Starting from \cref{lem:outer-minimal,lem:outer-reducible}, one can prove with a bit more work that an outer degree $\gamma\in\R_i(I)$ is irreducible if and only if it belongs to $\widetilde G(I)$ and, for every $\alpha\in\widetilde G(I)$ outer in the
direction $e_i^*$, the difference $\gamma-\alpha$ is not an outer root. However, this does not help to settle a minimal generating set, so we do not go into details.
\end{remark}

\cref{lem:outer-minimal,lem:outer-reducible} allow us however to gain something by reducing the outer roots to the set $\widetilde G(I)$. We sum up this partial reduction in the following result, which is an improvement of \cref{lem:discardable}.

\begin{proposition}\label{prop:discardable2}
The Lie algebra $\mathfrak g$ is generated by $\mathfrak t\cup\bigcup_{\alpha}\mathfrak g_\alpha$, where $\alpha$ ranges over $\widetilde G(I)$, all inner roots $e_i\in E_\mathrm{min}$ for some set $E_\mathrm{min}$ as in \cref{lem:discardable}, and all irreducible roots $2e_i$.
\end{proposition}

\begin{proof}
Let $\gamma\in\mathcal R(I)$ be a root. It suffices to prove that $\mathfrak g_\gamma$ is generated by this set. Assume that $\gamma$ is inner. Then the result is immediate from (the proof of) \cref{lem:discardable} for all $\gamma$ except when $\gamma=2e_i$ is reducible. Indeed, if $\gamma=e_k$ for some $k$ then the outer roots involved belong to $\widetilde G(I)$ by \cref{rmk:outer-colon}, and in all other cases the decomposition in the induction process involves no outer roots, so passing from using all outer roots to only taking $\widetilde G(I)$ makes no difference. Assume then that $\gamma=2e_i$ is reducible. Then by \cref{prop:two-ei-criterion} and \cref{rmk:outer-colon}, we see that either $3e_i\in\supp(I)$, or $2e_i\in \supp(J_k)$ and $e_i+e_k\notin\supp(I)$, or $3e_i\in\supp(J_k)$ and $e_k\in\supp(J_i)$.

In all cases, we know from \cref{prop:two-ei-criterion} that $[\mathfrak g_{e_i},\mathfrak g_{e_i}]$ generates a subspace of $\mathfrak g_{2e_i}$ of codimension at most 1. In the first case, we actually know that $[\mathfrak g_{e_i},\mathfrak g_{e_i}]=\mathfrak g_{2e_i}$, so we are done. For the other two cases, it will suffice to prove that one can generate with brackets a derivation $\overline{\partial}_{2e_i,r}$ with $r(e_i)\neq0$.

In the second case, we have that $2e_i-e_k$, $e_i$ and $e_k$ are roots (the last two since $e_k\in E_{e_i}$ and $e_i\in E_{e_k}$). Then a direct computation gives $[\overline{\partial}_{2e_i-e_k,e_k^*},\overline{\partial}_{e_k,e_i^*}]=\overline{\partial}_{2e_i,e_i^*-2e_k^*}$.
Thus, if $2e_i-e_k\in\widetilde G(I)$ we are done. Otherwise, $\XX^{2e_i}\in J_k$ is not a generator and then we must have $\XX^{e_i}\in J_k$, in which case $e_i-e_k\in\widetilde G(I)$ is a root. Then \cref{lem:outer-reducible} tells us that $[\mathfrak g_{e_i-e_k},\mathfrak g_{e_i}]=\mathfrak g_{2e_i-e_k}$ and we can construct the same bracket as before.

In the third case, $3e_i-e_k$ is a root, and so are $e_k-e_i$ and $e_i$ (the latter since $e_i\in E_{e_i}$). Moreover, $\XX^{e_k}$ is clearly a generator of $J_i$, so that $e_k-e_i\in\widetilde G(I)$. Again a direct computation gives $[\overline{\partial}_{3e_i-e_k,e_k^*},\overline{\partial}_{e_k-e_i,e_i^*}]=\overline{\partial}_{2e_i,e_i^*-3e_k^*}$. Thus, if $3e_i-e_k\in\widetilde G(I)$ we are done. Otherwise, $\XX^{3e_i}\in J_k$ is not a generator of $J_k$ and then we must have either $\XX^{e_i}\in J_k$ or $\XX^{2e_i}\in J_k$. But then either automatically or by the argument from the previous case we get that $\mathfrak g_{2e_i-e_k}$ is already generated. Then \cref{lem:outer-reducible} gives again $[\mathfrak g_{2e_i-e_k},\mathfrak g_{e_i}]=\mathfrak g_{3e_i-e_k}$ and we conclude as above.\\

Assume now that $\gamma$ is outer. If $\gamma\in\widetilde G(I)$, the result is obvious. Assume then that $\gamma\notin\widetilde G(I)$ and that it is outer in the direction $e_i^*$. By \cref{lem:outer-minimal} we get a decomposition $\gamma=\alpha+\delta$ with $\alpha\in\R_i(I)$ and $\delta\in\ZZ_{\geq 0}^n$ nonzero. Note that $\delta_i=0$ since both $\alpha$ and $\gamma$ belong to $\R_i(I)$. Fix an index $k\neq i$ with $\alpha_k>0$, which exists because $s(\alpha)\geq 0$. Then $\delta$ is an inner root: indeed $\delta\leq\gamma+e_i=(\alpha+e_i)+\delta$ coordinate-wise and $\gamma+e_i\notin\supp(I)$, so $\delta\notin\supp(I)$; and $\alpha_k>0$ yields $\delta+e_k\leq\gamma+e_i$, so that $e_k\in E_\delta$. Then another direct computation yields
$[\overline{\partial}_{\alpha,e_i^*},\overline{\partial}_{\delta,e_k^*}]=\overline\partial_{\gamma,r}$ with $r=\delta_i e_k^*-\alpha_k e_i^*=-\alpha_k e_i^*$, which is nonzero. Hence the Lie bracket is nontrivial and generates the whole $1$-dimensional space $\mathfrak g_\gamma$. Since $\delta\neq0$, iterating strictly
decreases the summand in the componentwise order, so the process
terminates with a summand $\alpha$ in $\widetilde G(I)$ and (possibly several) summands in $\ZZ_{\geq 0}^n$.
\end{proof}

\subsection{General strategy} In order to state an actual algorithm for a minimal set of degrees (and, sometimes, generators), we need to impose first an order on the set of roots we will consider. The decomposition $G^0=R\rtimes L$ of
\cref{prop:semidirect,prop:levi} writes $G^0$ as a semidirect product, where
$L\cong\prod_{j=1}^r\GL_{m_j}(\KK)$ is the standard Levi factor and $R$ is the
unipotent radical, with
$$
  \operatorname{Lie}(R)=\bigoplus_{\gamma\in\Phi_R}\mathfrak g_\gamma,\qquad
  \Phi_R=\{\gamma\in\mathcal R(I): s(\gamma)>0\}
  \cup (\Sigma\smallsetminus\Sigma_L)=\R(I)\smallsetminus\Sigma_L,
$$
$\Sigma$ denoting the degree-sum-zero roots and $\Sigma_L\subseteq\Sigma$ those
occurring in opposite pairs.

\begin{lemma}\label{lem:order-PhiR}
There exists a linear functional $\Lambda\colon\ZZ^n\to\RR$ with
$\Lambda(\gamma)>0$ for every $\gamma\in\Phi_R$. Ordering $\Phi_R$ by the values
of $\Lambda$, and breaking ties by setting $e_i\in E_\mathrm{min}$ to be smaller than any other root with the same value not lying in $E_\mathrm{min}$, and arbitrarily otherwise, yields a total order $\geq_\Lambda$ such that:
\begin{itemize}
    \item $\Lambda(\delta)<\Lambda(\gamma)$ and hence $\delta<_\Lambda\gamma$ whenever $s(\delta)<s(\gamma)$ for roots $\delta,\gamma$;
    \item $\Lambda(\sigma)=0$ for every $\sigma\in\Sigma_L$;
    \item $e_i\leq_\Lambda \gamma$ for every $e_i\in E_\mathrm{min}$ and every $\gamma\in\Phi_R$ with $\Lambda(e_i)=\Lambda(\gamma)$;
    \item $i\alpha+j\beta>_\Lambda\max\{\alpha,\beta\}$ whenever $\alpha,\beta,\,i\alpha+j\beta\in\Phi_R$ and $i,j\geq1$.
\end{itemize}
\end{lemma}

\begin{proof}
Consider the relation ``$e_i\succ e_j$ and $e_j\not\succ e_i$'', that is, $e_i-e_j\in\Phi_R$. It is transitive by the computation recalled above: transitivity of
$\succ$ gives $e_i-e_k\in\R(I)$, and if $e_k-e_i$ were a root as well, then the same computation applied to $e_k-e_i$ and $e_i-e_j$ would make $e_k-e_j$ a root, contradicting $e_j-e_k\in\Phi_R$. It is moreover acyclic: a cycle would produce a pair of opposite degree-sum-zero roots, which lies in $\Sigma_L$ and not in $\Phi_R$. Finally, it is compatible with the blocks of $L$: if $e_i-e_j\in\Phi_R$ and $e_{i'}-e_i\in\Sigma_L$, then $i'\neq j$ (otherwise $e_j-e_i$ and $e_i-e_j$ would both be roots) and transitivity of $\succ$ gives $e_{i'}-e_j\in\R(I)$; moreover $e_j-e_{i'}\notin\R(I)$, since otherwise $e_j-e_i=(e_j-e_{i'})+(e_{i'}-e_i)$ would be a root, so $e_{i'}-e_j\in\Phi_R$. Hence $e_i-e_j\in\Phi_R$ induces a relation on the set of blocks, which is also transitive and acyclic.

Choose $\pi\colon\{1,\dots,n\}\to\ZZ$ with $\pi(i)>\pi(j)$ whenever $e_i-e_j\in\Phi_R$ and constant on each block, so that it is increasing along the relation on the set of blocks. Set $\Lambda(\gamma)=C\,s(\gamma)+\sum_{i=1}^n\pi(i)\gamma_i$ for $C$ sufficiently
large so that the first assertion holds. Then $\Lambda>0$ on every root of $\Phi_R$ with $s(\gamma)>0$, and $\Lambda(e_i-e_j)=\pi(i)-\pi(j)\geq 0$ on those of degree sum zero, with equality if and only if $e_i-e_j\in\Sigma_L$. This is the second assertion. The third assertion is immediate by construction and the last assertion is immediate from linearity.
\end{proof}

\begin{algorithm}\label{strategy:generators}
Given a monomial ideal $I$ with cofinite support, proceed as follows.
\begin{enumerate}
\item[\textup{Step~1.}] Compute the Levi factor $L$ from \cref{prop:levi} by computing the roots $\alpha$ with $s(\alpha)=0$, which is the set $\Sigma$. If $\Sigma_L=\Sigma\cap-\Sigma=\varnothing$, skip directly to Step~2. Otherwise, for each block $\GL_{m_j}(\KK)$ on indices $i_1,\dots,i_{m_j}$, retain the roots $\pm(e_{i_1}-e_{i_2}),\dots,\pm(e_{i_{m_j-1}}-e_{i_{m_j}})$ and the corresponding root groups $U_\alpha$.\\

The Lie subalgebra generated by the $\mathfrak g_\alpha$ for the retained roots $\alpha$ corresponds to the Lie algebra of the Levi subgroup $L$, since the root groups $U_{\pm(e_{i_1}-e_{i_2})},\dots,U_{\pm(e_{i_{m_j-1}}-e_{i_{m_j}})}$, together with $T$, generate $\GL_{m_j}(\KK)$. Moreover, such a generating subset is minimal for every block, so the whole subset of retained roots is also minimal.\\

\item[\textup{Step~2.}] Compute $\widetilde G(I)$ from the colon ideals of \cref{rmk:outer-colon}. The same remark allows us to compute the set $\mathcal R(I)_1:=\{\alpha\in\mathcal R(I)\mid s(\alpha)=1\}$ (computing inner roots in this set is obvious).

Then, as in \cref{lem:discardable}, compute a set $E_\mathrm{min}\subset E\subset\R(I)_1$ such that, for every $e_k\in E\cap\R(I)_1$ there exists $e_i\in E_\mathrm{min}$ such that $e_k\succ e_i$ or $e_k=e_i$, but there is no relation $e_j\succ e_i$ with $e_i,e_j\in E_{\mathrm{min}}$ distinct. This can be done by the transitivity of $\succ$, proved right before \cref{lem:discardable}, and it contains at most one root per factor of the Levi subgroup.

Finally, compute the set of irreducible roots of the form $2e_i$ using \cref{prop:two-ei-criterion}.

Define the set
\[X = (\widetilde G(I)\smallsetminus\Sigma_L)\;\cup\; E_{\mathrm{min}}\;\cup\;\{2e_i\mid \text{irreducible}\}.\]
It is a subset of $\Phi_R$, and hence admits the total order $\geq_\Lambda$ from \cref{lem:order-PhiR}. This is also the case for $\mathcal R(I)_1$.\\

By \cref{prop:discardable2}, we know that the subalgebra spanned by $\bigoplus_{\alpha\in \Sigma_L\cup X} \mathfrak g_\alpha$ is the whole Lie algebra $\mathfrak g$ of $G$. By Step~1, this is still the case if we replace $\Sigma_L$ by the currently retained subset of roots. In the following steps we will discard some subspaces from $\bigoplus_{\alpha\in X} \mathfrak g_\alpha$.\\

\item[\textup{Step~3.}] Using the total order $\geq_\Lambda$ from \cref{lem:order-PhiR}, and assuming that we already have subsets $X_r,X_d\subseteq X$ of retained and discarded roots (both are empty at the end of Step 2), consider the smallest remaining element $\alpha\in X$ (i.e.~$\alpha\not\in X_r\cup X_d$). Then:
\begin{itemize}
    \item If $\alpha$ is outer, check whether it decomposes as a sum $\alpha=\alpha_0+\beta$ with $\alpha_0,\beta$ outer roots (this can be checked easily from \cref{rmk:outer-colon}, and no need for inner roots by \cref{lem:outer-minimal}) and $\alpha_0\in \Sigma_L\cup X$ smaller in the $\Lambda$-order. If $\beta$ is not of this form, seek recursively to decompose further (eventually with inner factors if $\beta\not\in\widetilde G(I)$, by \cref{lem:outer-minimal}) until all summands are in $\Sigma_L\cup X$ and smaller in the $\Lambda$-order. If such a decomposition exists, discard $\alpha$. Otherwise retain it and retain the root subgroup $U_\alpha$.
    \item If $\alpha=e_i$, compute the subspace $B^{\Lambda}_{e_i}\subseteq B_{e_i}$ spanned by those decompositions of \cref{prop:ei-Iouter-search} of types \textup{(B)} and \textup{(C)} whose summands are either in $\Sigma$ or are in $\mathcal R(I)_1$ and are smaller in the $\Lambda$-order.
    Choose a complement of it in $\mathfrak g_{e_{i}}$ (which is nonempty by \cref{prop:ei-Iouter-search}), fix a basis of this complement and retain the root subgroups $U_{e_i,p}$ for $p$ in the basis.
    \item If $\alpha=2e_i$, retain the root subgroup $U_{2e_i,e_i^*}$ and nothing else.
\end{itemize}

\item[\textup{Step~4.}] Iterate Step 3 until every degree in the set $X$ has been either retained or discarded, that is, until we have $X=X_r\sqcup X_d$.
\end{enumerate}
\end{algorithm}

\begin{theorem}\label{thm:minimal-generators}
Let $I$ be a monomial ideal with cofinite support. Let $G$ be the linear algebraic $\KK$-group of automorphisms of $\KK[\XX]/I$ and let $G^0$ be its neutral connected component. Then $G^0$ is generated by the maximal torus $T$ that acts on the variables $\overline x_i$ by scalars, and the subgroups $U_{\alpha,p}$ given by \cref{strategy:generators}.

Moreover, the set of degrees $\alpha$ is minimal, in the sense that, for any fixed $\alpha$, if we take away all the $U_{\alpha,p}$, then the set of subgroups ceases to generate $G^0$. And if the Levi subgroup of $G^0$ given by \cref{prop:levi} is equal to $T$ (equivalently, if $\Sigma_L=\varnothing$), then the set is minimal in the stronger sense that one cannot take away any single $U_{\alpha,p}$.

In particular, if $\mathfrak g,\,\mathfrak t,\,\mathfrak g_{\alpha,p}$ denote the Lie algebras of $G^0,\,T,\, U_{\alpha,p}$ respectively, then $\mathfrak g$ is generated by $\mathfrak t$ and the $\mathfrak g_{\alpha,p}$ and their brackets can be computed using \cref{rmk:lie-bracket} and the natural action of the torus on $\mathfrak g_\alpha$.
\end{theorem}

\begin{proof}
The statement on the group $G^0$ follows immediately from the corresponding statement for $\mathfrak g$, so we focus on the latter.

By \cref{prop:discardable2}, we know that $\mathfrak t\oplus\bigoplus_{\alpha\in \Sigma_L\cup X} \mathfrak g_\alpha$ generates $\mathfrak g$. We claim that at each iteration of Step 3 the retained pieces generate the subalgebra $\bigoplus_{\alpha\in \Sigma_L\cup X_r\cup X_d} \mathfrak g_\alpha$. Since the algorithm stops when $X_r\cup X_d=X$, we get the whole algebra. To prove the claim, the only nontrivial part is for $\alpha=e_i$, where we use $\Sigma_L\cup\mathcal R(I)_1$ instead of $X$. In this case, we need to prove that the discarded subspace is generated by the previously retained derivations. For this it suffices to prove that $\mathfrak g_{\beta}$ is generated by the previously retained derivations for every $\beta\in\mathcal R(I)_1$ smaller than $e_i$ in the $\Lambda$-order. If $\beta$ is inner, we can assume this as an induction hypothesis. If $\beta$ is outer, by an iterated application of \cref{lem:outer-minimal,lem:outer-reducible}, we know that it can be written as a sum of (possibly several) elements in $\Sigma_L\cup X$ whose partial sums are all roots. And those summands in $X$ must be smaller in the $\Lambda$-order by \cref{lem:order-PhiR}. Then an iterated application of \cref{lem:outer-reducible} proves the claim.\\

Let us prove that the set of retained degrees $\alpha$ is minimal. Assume that there is a factor $\mathfrak g_\gamma$ with $\gamma\in X_r$ being generated via brackets by some other derivations in $\bigoplus_{\alpha\in (\Sigma_L\cup X_r)\smallsetminus\{\gamma\}} \mathfrak g_\alpha$. Then, for every such iterated bracket, we would have a decomposition $\gamma=\alpha_0+\alpha_1+\cdots+\alpha_m$ with $\alpha_i\in(\Sigma_L\cup X_r)\smallsetminus\{\gamma\}$ and $\beta_i:=\alpha_0+\cdots+\alpha_i\in\R(I)$ for every $i$. 

Assume first that $s(\gamma)=0$. If $\gamma>_\Lambda\alpha_i$ for every $i$, then $\gamma$ would have been discarded by construction. Consider then the first $i$ such that $\alpha_i>_\Lambda\gamma$. Then the partial sum $\beta_i$ would also be $>_\Lambda\gamma$ by the properties of $\Lambda$ (\cref{lem:order-PhiR}). But since $\Lambda$ is also linear, positive on $X_r$ and trivial on $\Sigma_L$, this implies that $\Lambda(\alpha_i)=\Lambda(\beta_i)=\Lambda(\gamma)$ and thus $\alpha_j\in\Sigma_L$ for every $j\neq i$. But then $\beta_{i-1}\in\Sigma_L$ since it is a root, so that $\alpha_i=\beta_i+(-\beta_{i-1})$ with $-\beta_{i-1}\in\Sigma_L$, and $\beta_{j-1}=\beta_j+(-\alpha_j)$ with $-\alpha_j\in\Sigma_L$ for every $j>i$. And since $\alpha_i>_\Lambda\gamma=\beta_m$ and $\gamma\in X$, it would have been discarded by the algorithm, so this case cannot occur.

We know then that the set of retained roots with $s(\gamma)=0$ is minimal as a generating set for $\bigoplus_{\alpha\in\Sigma} \mathfrak g_\alpha$. Since clearly no other root can contribute to this subalgebra, we may assume in what follows that $s(\gamma)>0$ and that the $\alpha_i$ lie in $(\Sigma\cup X_r)\smallsetminus\{\gamma\}$. In that case, if $\alpha_j$ is the first summand with $s(\alpha_j)>0$, we see that $\beta_{j-1}\in\Sigma$ since it is a root. But then we may assume that the first bracket is between $\beta_{j-1}$ and $\alpha_j$. And since brackets anticommute, we may swap these two and assume henceforth that $s(\alpha_0)>0$.

Assume that $\gamma$ is outer. Again, if $\gamma>_\Lambda\alpha_i$ for every $i$, then $\gamma$ would have been discarded by construction. This happens in particular if $\Lambda(\gamma)>\Lambda(\alpha_0)$, and thus we must have $\Lambda(\gamma)=\Lambda(\alpha_0)$ (it cannot be lesser by linearity) and thus $\alpha_i\in\Sigma_L$ for every $i>0$.  But then it must be $\alpha_0>_\Lambda\gamma$ and $\beta_{i-1}=\beta_i+(-\alpha_i)$ with $-\alpha_i\in\Sigma_L$ for every $i>0$. So, if $\alpha_0$ is outer, then it would have been discarded by the algorithm since $\gamma=\beta_m\in X$. If $\alpha_0$ is inner, it cannot be of the form $e_k\in E_\mathrm{min}$ since $\Lambda(e_k)=\Lambda(\gamma)$ implies $e_k<_\Lambda\gamma$. And if $\alpha_0=2e_k$, the decomposition $2e_k=\beta_1+(-\alpha_1)$ implies that $\alpha_0$ is reducible by (the proof of) \cref{prop:two-ei-criterion}, so this case is also impossible.

Assume that $\gamma$ is inner. It cannot be of the form $2e_k$ since it is irreducible. Then $\gamma=e_k\in E_{\mathrm{min}}$ and hence $s(\alpha_0)=1$ and $s(\alpha_i)=0$ for $i>0$. Recall from \cref{prop:ei-Iouter-search} that the image of decompositions of types \textup{(B)} and \textup{(C)} is always contained in $D_{e_i}=\{\overline\partial_{e_i,p}\in\mathfrak g_{e_i}\mid p(e_i+\mathbf{1})=0\}$ and that the image of $D_{e_i}$ via a decomposition of type \textup{(A)} is precisely $D_{e_k}\subsetneq \mathfrak g_{e_k}$, where $e_k$ is the target root. Let us prove that the image of the iterated bracket coming from the decomposition into $\alpha_i$'s lies in this subspace, which yields that one does not generate the whole $\mathfrak g_{e_k}$ with these decompositions. If any of the $\beta_i$ is an outer root, then we are done since then the image of every subsequent decomposition $\beta_j=\beta_{j-1}+\alpha_j$ with $\beta_j$ inner will fall into the corresponding subspace. So every $\beta_i$ must be inner and every decomposition $\beta_{i}=\beta_{i-1}+\alpha_i$ must be of type \textup{(A)}. But then we get $\beta_{i+1}\succ\beta_i$ for every $i$, and hence $e_k=\beta_m\succ \beta_0=\alpha_0$, which is impossible since $e_k\in E_{\mathrm{min}}$.

Finally, let us prove that the set of derivations is minimal if $L=T$, which is equivalent to $\Sigma_L=\varnothing$. The only derivations that could still be discarded are in $\mathfrak g_{e_i}$ for some $e_i\in X_r$, since for outer roots and inner roots of the form $2e_j$ we have retained a 1-dimensional subspace of $\mathfrak g_\alpha$. Now, the algorithm has already discarded every derivation arising from a decomposition of type \textup{(B)} or \textup{(C)} whose summands are smaller in the $\Lambda$-order. By the properties of $\Lambda$, there cannot be such a decomposition with a summand that is greater in the $\Lambda$-order. We are only left then with possible decompositions of type \textup{(A)}. But these are discarded as well by the argument above since $e_i\in E_\mathrm{min}$.
\end{proof}

\begin{remark}
The case $L=T$, or $\Sigma_L=\varnothing$, is indeed the generic case. In a randomly chosen monomial ideal, one should not be able to permute two variables $\overline x_i$ and $\overline x_j$ (in particular, there are no toric automorphisms). But any nontrivial block in the Levi subgroup yields variables that can be exchanged (and much more!).
\end{remark}

\begin{example}\label{ex:strategy-run}
We follow \cref{strategy:generators} step by step on the ideals of
\cref{ex:outer-colon,ex:two-ei,ex:ei-partial,ex:m2-outer-outer}. Throughout, $C$
denotes a sufficiently large constant and
$\Lambda(\gamma)=C\,s(\gamma)+\sum_{i}\pi(i)\gamma_i$ is the functional of
\cref{lem:order-PhiR}, for a choice of $\pi$ made explicit in each case.

\medskip\noindent\emph{(a) $I=(x_1^2,\,x_1x_2,\,x_2^3)$, the ideal of \cref{ex:outer-colon}.}
Here $\R(I)=\{e_1-e_2,\ -e_1+2e_2,\ e_2\}$, the first two being outer, and every
$\mathfrak g_\alpha$ is one-dimensional, so $\dim G^0=2+3=5$.

\emph{Step~1.} The only root of degree sum zero is $e_1-e_2$, so $\Sigma=\{e_1-e_2\}$ and
$\Sigma_L=\Sigma\cap-\Sigma=\varnothing$. Thus $L=T$ and we pass directly to Step~2.

\emph{Step~2.} By \cref{ex:outer-colon}, $\widetilde G(I)=\{-e_1+2e_2,\ e_1-e_2\}$ and
$\R(I)_1=\{-e_1+2e_2,\ e_2\}$. Here $E\cap\R(I)_1=\{e_2\}$, since $e_1$ is not a root, so
$E_{\mathrm{min}}=\{e_2\}$. Neither $2e_1$ nor $2e_2$ is a root, since $2e_1\in\supp(I)$ and
$E_{2e_2}=\varnothing$. Therefore
$$X=\{e_1-e_2,\ -e_1+2e_2,\ e_2\}.$$
The order forces $\pi(1)>\pi(2)$, so we take $\pi=(1,0)$; then
$\Lambda(e_1-e_2)=1$, $\Lambda(-e_1+2e_2)=C-1$ and $\Lambda(e_2)=C$, which is the order in
which Steps~3--4 process $X$.

\emph{Steps~3--4.} The degree $e_1-e_2$ admits no decomposition into outer roots at all, and
neither does $-e_1+2e_2$: the only candidate splitting would involve $-e_1+e_2$, which is not
a root. Both are retained. Finally, $E_{e_2}=\{e_2\}$, so $\dim\mathfrak g_{e_2}=1$, and the
type-\textup{(B)} decomposition
$$e_2=(e_1-e_2)+(2e_2-e_1)$$
has its first summand in $\Sigma$ and its second in $\R(I)_1$ with
$\Lambda(2e_2-e_1)=C-1<C=\Lambda(e_2)$, so it is taken into account; by
\cref{prop:ei-Iouter-search} it contributes $\overline\partial_{e_2,\,2e_1^*-e_2^*}$, whose
class in $\mathfrak g_{e_2}$ is nonzero. Hence $B^\Lambda_{e_2}=\mathfrak g_{e_2}$ and $e_2$
is discarded.

The algorithm returns
$$U_{e_1-e_2},\qquad U_{-e_1+2e_2}.$$
As $\Sigma_L=\varnothing$, \cref{thm:minimal-generators} applies in its strong form: neither
of the two subgroups can be removed.

\medskip\noindent\emph{(b) $I=(x_1^4,\,x_1^3x_2,\,x_2^3)$, the ideal of \cref{ex:two-ei}.}
Here $\R(I)$ consists of the two outer roots $-e_1+2e_2$ and $3e_1-e_2$ and the seven inner
roots $e_1,e_2,e_1+e_2,2e_1,2e_2,e_1+2e_2,2e_1+e_2$; one computes $\dim G^0=2+13=15$.

\emph{Step~1.} There is no root of degree sum zero, so $\Sigma=\Sigma_L=\varnothing$, $L=T$,
and we pass to Step~2.

\emph{Step~2.} The colon ideals give $\widetilde G_1=\{x_2^2\}$ and $\widetilde G_2=\{x_1^3\}$,
so $\widetilde G(I)=\{-e_1+2e_2,\ 3e_1-e_2\}$, while $\R(I)_1=\{-e_1+2e_2,\ e_1,\ e_2\}$. Both
$e_1$ and $e_2$ are roots and, as $\Sigma=\varnothing$, the relation $\succ$ is empty on
$E\cap\R(I)_1=\{e_1,e_2\}$; hence $E_{\mathrm{min}}=\{e_1,e_2\}$. By \cref{ex:two-ei}, $2e_1$
is irreducible while $2e_2$ is not. Therefore
$$X=\{-e_1+2e_2,\ e_2,\ e_1,\ 2e_1,\ 3e_1-e_2\}.$$
Since $\succ$ is empty, $\pi$ is unconstrained; taking $\pi=(2,1)$ orders $X$ as displayed,
with $\Lambda$-values $C$, $C+1$, $C+2$, $2C+4$ and $2C+5$.

\emph{Steps~3--4.} Neither $-e_1+2e_2$ nor $3e_1-e_2$ decomposes as a sum of outer roots, so
both are retained. For the two degrees $e_i$ we have $E_{e_1}=E_{e_2}=\{e_1,e_2\}$, so
$\dim\mathfrak g_{e_1}=\dim\mathfrak g_{e_2}=2$; moreover $\Sigma=\varnothing$, so $e_1$ and
$e_2$ admit no decomposition of type \textup{(B)} or \textup{(C)} whatsoever and
$B^\Lambda_{e_1}=B^\Lambda_{e_2}=0$: both planes are retained in full. Finally $2e_1$ is
irreducible and contributes $U_{2e_1,e_1^*}$.

The algorithm returns the seven subgroups
$$U_{-e_1+2e_2},\quad U_{3e_1-e_2},\quad U_{2e_1,e_1^*},\quad
U_{e_1,e_1^*},\quad U_{e_1,e_2^*},\quad U_{e_2,e_1^*},\quad U_{e_2,e_2^*}.$$
Again $\Sigma_L=\varnothing$, so none of them can be removed.

\medskip\noindent\emph{(c) $I=\M^2\subset\KK[x_1,x_2,x_3]$, the ideal of
\cref{ex:m2-outer-outer}.} Here $\R(I)=\{e_i-e_j\mid i\neq j\}$ and $\dim G^0=3+6=9$.

\emph{Step~1.} $\Sigma=\Sigma_L=\R(I)$, giving a single block $\GL_3(\KK)$ on
$\{1,2,3\}$; we retain $\pm(e_1-e_2)$ and $\pm(e_2-e_3)$ together with the four
corresponding root groups.

\emph{Step~2.} By \cref{ex:m2-outer-outer}, $\widetilde G(I)=\R(I)=\Sigma_L$, so
$\widetilde G(I)\smallsetminus\Sigma_L=\varnothing$. Here $\R(I)_1=\varnothing$; in
particular no $e_i$ is a root, so $E\cap\R(I)_1=\varnothing$ and
$E_{\mathrm{min}}=\varnothing$. No $2e_i$ is a root either. Hence $X=\varnothing$.

\emph{Steps~3--4.} Nothing to do.

The algorithm returns $U_{\pm(e_1-e_2)}$ and $U_{\pm(e_2-e_3)}$, the standard set of root
subgroups generating $G^0=\GL_3(\KK)$, in accordance with
\cref{ex:reducible-not-discardable}. This is the case $d=2$ of the family
$\KK[\XX]/\M^d$, whose roots are described in general in \cref{lem:anick-roots}.

This family is minimal, but only in the sense of inclusion: removing any one of the
four subgroups leaves a subalgebra of dimension $7$, yet the three subgroups
$U_{e_1-e_2}$, $U_{e_2-e_3}$, $U_{e_3-e_1}$ already generate, together with $\mathfrak t$, the whole
$\mathfrak{gl}_3(\KK)$. So the output of Step~1 need not have least cardinality among the
generating families of root subgroups; what \cref{thm:minimal-generators} asserts is that no
proper subfamily of the returned one generates.

\medskip\noindent\emph{(d) $I=(x_1^2,\,x_2^2,\,x_3^2,\,x_1x_3)$, the ideal of
\cref{ex:ei-partial}.} Here
\[\R(I)=\{\pm(e_1-e_3),\ e_1,\ e_2,\ e_3,\ e_1+e_2-e_3,\ -e_1+e_2+e_3\},\]
with $\dim\mathfrak g_{e_2}=2$ and all the other root spaces of dimension one, so
$\dim G^0=3+8=11$.

\emph{Step~1.} $\Sigma=\Sigma_L=\{\pm(e_1-e_3)\}$, which gives a single nontrivial block
$\GL_2(\KK)$ on the indices $\{1,3\}$. We retain $\pm(e_1-e_3)$ and the root groups
$U_{e_1-e_3}$ and $U_{e_3-e_1}$.

\emph{Step~2.} From \cref{ex:ei-partial}, $J_1=J_3=(x_1,x_2^2,x_3)$ and
$J_2=(x_1^2,x_1x_3,x_2,x_3^2)$, so that $\widetilde G_1=\{x_3\}$, $\widetilde G_3=\{x_1\}$,
$\widetilde G_2=\varnothing$ and $\widetilde G(I)=\{\pm(e_1-e_3)\}\subseteq\Sigma_L$; thus
$\widetilde G(I)\smallsetminus\Sigma_L=\varnothing$. Moreover
$\R(I)_1=\{e_1,\ e_2,\ e_3,\ e_1+e_2-e_3,\ -e_1+e_2+e_3\}$. All three $e_i$ are roots and the
relation $\succ$ on $E\cap\R(I)_1=\{e_1,e_2,e_3\}$ consists of $e_1\succ e_3$ and
$e_3\succ e_1$, so its classes are $\{e_1,e_3\}$ and $\{e_2\}$, both minimal, and
$E_{\mathrm{min}}$ picks one element from each: we take $E_{\mathrm{min}}=\{e_2,e_3\}$
(choosing $e_1$ instead of $e_3$ works symmetrically). No $2e_i$ is a root. Therefore
$$X=\{e_2,\ e_3\}.$$
Since $\Sigma\smallsetminus\Sigma_L=\varnothing$, the only constraint on $\pi$ is that it be
constant on the block $\{1,3\}$; we take $\pi=(1,0,1)$, so that $\Lambda(e_2)=C$ and
$\Lambda(e_1)=\Lambda(e_3)=C+1$.

\emph{Steps~3--4.} For $e_2$ we have $E_{e_2}=\{e_1,e_3\}$ and, by \cref{ex:ei-partial}, its
only decompositions are the two of type \textup{(C)}
$$e_2=(e_1-e_3)+(e_2-e_1+e_3)=(e_3-e_1)+(e_2+e_1-e_3).$$
In both, the first summand lies in $\Sigma$ and the second lies in $\R(I)_1$, so what decides
is whether the second summand precedes $e_2$ in the $\Lambda$-order. It does not: the two
second summands differ from $e_2$ by an element of $\Sigma_L$, on which $\Lambda$ vanishes, so
that
$$\Lambda(e_2-e_1+e_3)=\Lambda(e_2+e_1-e_3)=\Lambda(e_2)=C$$
for every $\pi$ allowed by \cref{lem:order-PhiR}, and the tie-breaking of
\cref{lem:order-PhiR} places $e_2\in E_{\mathrm{min}}$ before both. Hence
$B^{\Lambda}_{e_2}=0$ and the whole plane $\mathfrak g_{e_2}$ is retained, contributing the
two subgroups $U_{e_2,e_1^*}$ and $U_{e_2,e_3^*}$.

For $e_3$ we have $E_{e_3}=\{e_2\}$, so $\dim\mathfrak g_{e_3}=1$. Its only decomposition is
the type-\textup{(A)} one $e_3=(e_3-e_1)+e_1$, which does not enter $B^\Lambda_{e_3}$, only
the types \textup{(B)} and \textup{(C)} being used there. Thus $B^{\Lambda}_{e_3}=0$ and we
retain $U_{e_3,e_2^*}$.

The algorithm returns the five subgroups
$$U_{e_1-e_3},\quad U_{e_3-e_1},\quad U_{e_2,e_1^*},\quad U_{e_2,e_3^*},\quad U_{e_3,e_2^*}.$$
Note that $e_1\notin X$: the space $\mathfrak g_{e_1}$ is recovered afterwards from the
type-\textup{(A)} decomposition $e_1=(e_1-e_3)+e_3$. This is precisely the circularity of
\cref{ex:reducible-not-discardable}: both $e_1$ and $e_3$ are reducible, but they recover each
other, and since $\Lambda(e_1)=\Lambda(e_3)$ it is the choice of $E_{\mathrm{min}}$ in Step~2,
not the order $\Lambda$, that decides which one survives.

This example also shows that the family returned by \cref{strategy:generators} need
not be minimal as a family of subgroups, so that the last assertion of
\cref{thm:minimal-generators} really does require $\Sigma_L=\varnothing$. Its set of
\emph{degrees} is minimal: discarding all the subgroups attached to any one of $e_1-e_3$,
$e_3-e_1$, $e_2$ or $e_3$ leaves a proper subgroup of $G^0$. But the two retained directions
of $\mathfrak g_{e_2}$ are not independent of one another, since by \cref{rmk:lie-bracket}
$$\bigl[\,\overline\partial_{e_1-e_3,\,e_3^*},\,
\bigl[\overline\partial_{e_3-e_1,\,e_1^*},\,\overline\partial_{e_2,\,e_1^*}\bigr]\,\bigr]
=\overline\partial_{e_2,\,e_1^*-e_3^*},$$
so that $U_{e_2,e_1^*}$ together with the two root subgroups of the Levi factor already
recovers $\overline\partial_{e_2,e_3^*}$, and symmetrically with the roles of $e_1^*$ and
$e_3^*$ exchanged. Either one of $U_{e_2,e_1^*}$ and $U_{e_2,e_3^*}$ may therefore be removed, although not both, and the resulting family of four subgroups still generates $G^0$.
The algorithm cannot see this: the two decompositions of $e_2$ that would reveal it are
exactly the two type-\textup{(C)} ones discarded above.
\end{example}

\section{An application: Anick's theorem}\label{sec:anick}

In this last section we apply the description of \cref{sec:generators} to the algebras
$\KK[\XX]/\M^d$, where $\M=(x_1,\dots,x_n)\subset\KK[\XX]$, and we deduce Anick's theorem
on the density of tame automorphisms \cite{Anick83}. Throughout, $n\geq2$ and $d\geq2$; we
write $A_d=\KK[\XX]/\M^d$ and $G=\Aut_\KK(A_d)$, and we set
$\mathbf 1=e_1+\cdots+e_n\in M$. Note that $\M^d$ has cofinite support and that
$x_i\notin\supp(\M^d)$, so all the results above apply.

Recall that the tame subgroup of $\Aut_\KK(\KK[\XX])$ is the subgroup generated by
the affine automorphisms and by the elementary ones, that is those of the form
\[x_j\mapsto x_j+f(x_1,\dots,\widehat{x_j},\dots,x_n),\qquad
x_i\mapsto x_i\;\text{ if } i\neq j,\]
where $\widehat{x_j}$ means that the variable $x_j$ is being omitted.

\begin{lemma}\label{lem:anick-roots}
The inner roots of $\M^d$ are $\{\alpha\in\ZZ^n_{\geq 0}\ :\ 1\leq s(\alpha)\leq d-2\}$, while the outer roots are given by $\R_j(\M^d)=\{\beta-e_j\ :\ \beta\in\ZZ^n_{\geq0},\ \beta_j=0,\ 1\leq s(\beta)\leq d-1\}$, with $1\leq j\leq n$.
Moreover:
\begin{enumerate}[(i)]
\item $E_\alpha=E$ for every inner root $\alpha$, so that $N(\alpha)=N_\KK$ and
      $\dim\mathfrak g_\alpha=n$;
\item $\Sigma=\Sigma_L=\{e_i-e_j\mid i\neq j\}$, so that $L\cong\GL_n(\KK)$ and $G=G^0$;
\item every inner root is reducible.
\end{enumerate}
\end{lemma}

\begin{proof}
We have $\supp(\M^d)=\{\mathbf{m}\in\ZZ^n_{\geq0}: s(\mathbf{m})\geq d\}$ and hence, by definition,
\begin{align*}
E_\alpha &=\{e_i\in E\mid \alpha+e_i\in\ZZ^n_{\geq0}\text{ and } s(\alpha+e_i)\leq d-1\}\\
&=\{e_i\mid \alpha+e_i\in\ZZ^n_{\geq0}\text{ and } s(\alpha)\leq d-2\}.
\end{align*}
If $\alpha$ is inner, then $\alpha+e_i\in\ZZ^n_{\geq0}$ for every $i$, so $E_\alpha$ is
either empty or all of $E$, and it is nonempty exactly when $s(\alpha)\leq d-2$. This
proves the description of the inner roots and \textup{(i)}.

Now, let $\alpha$ be a degree with $\alpha_j=-1$ and $\alpha_i\geq0$ for $i\neq j$. Condition \eqref{eq:outer-iii} for outer roots holds if and only if $s(\alpha)\geq0$ since $s(\mathbf m+\alpha)=s(\mathbf m)+s(\alpha)$. On the other hand, $E_\alpha\neq\varnothing$ if and only if $e_j\in E_\alpha$, that is, if and only if $s(\alpha)\leq d-2$. Writing
$\beta=\alpha+e_j$ turns these two conditions into $1\leq s(\beta)\leq d-1$, which gives
the description of the outer roots.

For \textup{(ii)}, the degrees with sum zero are the $e_i-e_j$ with $i\neq j$, and all of them are outer roots by the description above. Hence
$\Sigma=\Sigma_L$ and \cref{prop:levi} yields $L\cong\GL_n(\KK)$. In particular the toric automorphisms, which act on $\M/\M^2$ by permutation matrices, lie in $L\subseteq G^0$. And since they generate $G/G^0$
by \cref{thm:dlmr}, we get $G=G^0$.

Finally, let $\alpha$ be an inner root and let us see that it is reducible. By \cref{lem:recovery-mixed,lem:recovery-multiple}, we only need to check this for $\alpha=e_i$ or $\alpha=2e_i$. If $\alpha=2e_i$, either $d\leq3$ and $2e_i$ is not even a root; or $d\geq 4$ and then by the description above $2e_i-e_k$ is an outer root and $e_i+e_k\notin\supp(\M^d)$ for any $k\neq i$, so case~\textup{(ii)} of \cref{prop:two-ei-criterion} applies. If $\alpha=e_i$, then $e_i-e_k$ is an outer root and $e_k$ is an inner root for any $k\neq i$, so $e_i$ admits a decomposition of type
\textup{(A)} and \cref{prop:ei-Iouter-search} gives $B_{e_i}=\mathfrak g_{e_i}$.
\end{proof}

\begin{remark}\label{rmk:anick-not-discardable}
As \cref{ex:reducible-not-discardable} warns, \cref{lem:anick-roots}\,\textup{(iii)} does
not say that the inner root subgroups may all be discarded. In fact they cannot as soon as $d\geq3$ (for $d=2$ there are no inner roots): the
relation $\succ$ is complete on $E$, so $E_{\mathrm{min}}$ consists
of a single $e_i$; and since every decomposition of $e_i$ of type \textup{(B)} or \textup{(C)} has a summand with the same $\Lambda$-value as $e_i$, which is larger than $e_i$ in the $\Lambda$-order by \cref{lem:order-PhiR}, none of them is taken into account and
\cref{strategy:generators} retains the whole of $\mathfrak g_{e_i}$, a space of dimension $n$
by \cref{lem:anick-roots}\,\textup{(i)}. What we prove below is that the subgroup generated by the torus and the
outer root subgroups is exactly the subgroup of automorphisms with constant Jacobian
determinant, which does not account for the whole $\mathfrak g_{e_i}$, but it is all that the application requires.
\end{remark}

\medskip

We now make the Jacobian operator available on $A_d$. Every class in $A_d$ has a unique
representative of degree $<d$; given $\varphi\in G$, let $F_\varphi$ be the endomorphism of $\KK[\XX]$ induced by the $n$-tuple $(F_1,\dots,F_n)$ of representatives of degree $<d$ of the $\varphi(\overline x_i)$, and put
$$
\jac(\varphi)\ :=\ \det\Bigl(\frac{\partial F_i}{\partial x_j}\Bigr)_{i,j}
\ \in\ \KK[\XX]/\M^{d-1}.
$$

This is well defined: changing an $F_i$ by an element of $\M^d$ changes each entry of the
matrix by an element of $\M^{d-1}$, hence the determinant as well. In the same way we get a
well-defined divergence operator
$$
\divi\colon \Der(A_d)\longrightarrow \KK[\XX]/\M^{d-1},\qquad
\divi\Bigl(\sum_i f_i\frac{\partial}{\partial x_i}\Bigr)=\sum_i\frac{\partial f_i}{\partial x_i},
$$
computed on any lift to $\Der(\KK[\XX])$. Since
$\partial_{\alpha,p}=\sum_i p(e_i)\,\XX^{\alpha+e_i}\,\partial/\partial x_i$, one gets immediately
\begin{equation}\label{eq:divergence}
\divi(\overline\partial_{\alpha,p})=p(\alpha+\mathbf 1)\,\XX^{\alpha}
\quad\text{for $\alpha$ inner},
\qquad
\divi(\overline\partial_{\alpha,e_j^*})=0\quad\text{for $\alpha\in\R_j(\M^d)$},
\end{equation}
the second equality because $(\alpha+e_j)_j=0$.

\begin{lemma}\label{lem:anick-jac}
With notations as above:
\begin{enumerate}[(i)]
\item $\jac(\varphi\psi)=\jac(\varphi)\cdot\bigl(\jac(\psi)\circ F_\varphi\bigr)$ for every $\varphi,\psi\in G$. Therefore $G^{\jac}:=\{\varphi\in G\mid\jac(\varphi)\in\KK^\times\}$ is a subgroup of $G$ and $\jac\colon G^{\jac}\to\KK^\times$ is a group homomorphism.
\item If $\overline{\partial}\in\operatorname{Lie}(U_1)$ has lowest nonzero homogeneous component $\overline{\partial}_k\in \bigoplus_{s(\alpha)=k}\mathfrak g_\alpha$, $k\geq 1$, then $\jac(\exp \overline{\partial})=1+\divi(\overline{\partial}_k)+(\text{terms of higher degree})$.
\end{enumerate}
\end{lemma}

\begin{proof}
\textup{(i)} If $F=F_\varphi$ and $H=F_\psi$, then $F_{\varphi\psi}$ is the truncation of $H\circ F$ in degrees $<d$, and the chain rule gives
$\det J_{H\circ F}=(\det J_H\circ F)\cdot\det J_F$. Since $F_i\in\M$, $F$ maps $\M^{d-1}$ into itself and therefore descends to $\KK[\XX]/\M^{d-1}$, which yields the
stated formula. If $\jac(\varphi)=c$ and $\jac(\psi)=c'$ with $c,c'\in\KK$, the formula gives $\jac(\varphi\psi)=cc'$; and applying it to $\psi=\varphi^{-1}$ gives
$(\jac(\varphi^{-1})\circ F)\,c=1$, so $\jac(\varphi^{-1})$ is constant since $F$ is an automorphism of
$\KK[\XX]/\M^{d-1}$.

\textup{(ii)} Since $\overline{\partial}$ is nilpotent, $F_{\exp \overline{\partial}}$ is the truncation of
$x_i+\partial(x_i)+\tfrac12\partial^2(x_i)+\cdots$. The term $\partial(x_i)$ has lowest degree $k+1$ and
$\partial^{\ell}(x_i)$ has lowest degree $\ell k+1>k+1$ for $\ell\geq2$, so
$J_{F_{\exp \overline{\partial}}}=\mathrm{Id}+J_{\overline{\partial}_k}+(\text{terms of degree}>k)$. Expanding the determinant,
every summand other than the product of the diagonal entries involves at least two
off-diagonal entries and so has degree $\geq2k>k$; hence
$\det=1+\operatorname{tr}J_{\overline{\partial}_k}+(\text{terms of degree}>k)$, and
$\operatorname{tr}J_{\overline{\partial}_k}=\divi(\overline{\partial}_k)$.
\end{proof}

\begin{proposition}\label{prop:anick-outer}
Let $H\subseteq G$ be the subgroup generated by the maximal torus $T$ and the root
subgroups $U_\alpha$, with $\alpha\in\R(\M^d)$ outer. Then $H=G^{\jac}$. Moreover
$$\operatorname{Lie}(H)=\mathfrak g_0\oplus\bigoplus_{\alpha\text{ outer}} \mathfrak g_\alpha\oplus\bigoplus_{\alpha\text{ inner}}\{\overline{\partial}\in\mathfrak g_\alpha\mid \divi(\overline{\partial})=0\},$$
so that $H$ has codimension $\binom{n+d-2}{n}-1$ in $G$, the number of inner roots.
\end{proposition}

\begin{proof}
We first prove $H\subseteq G^{\jac}$. The torus acts by $\overline x_i\mapsto t_i\overline x_i$, so $\jac=t_1\cdots t_n\in\KK^\times$. If $\alpha\in\R_j(\M^d)$ and
$\beta=\alpha+e_j$, then $\partial_{\alpha,e_j^*}=\XX^{\beta}\partial/\partial x_j$ with $\beta_j=0$, so that $\exp(t\,\overline\partial_{\alpha,e_j^*})$ satisfies
\[\overline{x}_j\mapsto \overline{x}_j+t\XX^{\beta},\qquad \overline{x}_i\mapsto \overline{x}_i,\, \text{ for }i\neq j,\]
and hence its Jacobian matrix is unipotent triangular; so $\jac=1$. As $G^{\jac}$ is a subgroup by \cref{lem:anick-jac}\,\textup{(i)}, we get $H\subseteq G^{\jac}$.

Next we compute $\mathfrak h:=\operatorname{Lie}(H)$. From the identity
$\divi([\partial,\partial'])=\partial(\divi \partial')-\partial'(\divi \partial)$, formula \eqref{eq:divergence} and $[\mathfrak g_0,\mathfrak g_\alpha]=\mathfrak g_\alpha$, we deduce that $\divi\overline{\partial}=0$ for every $\overline{\partial}\in\mathfrak h\cap\mathfrak g_\alpha$ with $\alpha\neq 0$. Then, again by \eqref{eq:divergence}, for every inner root $\alpha$ we have
\[\mathfrak h\cap\mathfrak g_\alpha\subseteq D_{\alpha}=\{\overline\partial_{\alpha,p}\mid p(\alpha+\mathbf 1)=0\}=\{\overline{\partial}\in\mathfrak g_\alpha\mid \divi(\overline{\partial})=0\}.\]
We claim that this is an equality. Fix an inner root $\alpha$ and an index $j$, and for
each $i\neq j$ set
$$
\beta_i=\alpha+e_i-(\alpha_j+1)e_j,\qquad \gamma_i=(\alpha_j+1)e_j-e_i.
$$
Then $\beta_i\in\R_j(\M^d)$ and $\gamma_i\in\R_i(\M^d)$: indeed $(\beta_i)_j=-1$ and the
remaining coordinates of $\beta_i$ are $\geq0$, with
$s(\beta_i)=s(\alpha)-\alpha_j\in[0,d-2]$, and likewise $(\gamma_i)_i=-1$,
$(\gamma_i)_j=\alpha_j+1\geq1$ and $s(\gamma_i)=\alpha_j\in[0,d-2]$. By
\cref{rmk:lie-bracket},
$$
\bigl[\overline\partial_{\beta_i,\,e_j^*},\ \overline\partial_{\gamma_i,\,e_i^*}\bigr]
=\overline\partial_{\alpha,\ (\alpha_j+1)e_i^*-(\alpha_i+1)e_j^*}\ \in\ \mathfrak h.
$$
The $n-1$ covectors $(\alpha_j+1)e_i^*-(\alpha_i+1)e_j^*$, for $i\neq j$, are linearly
independent, since each involves $e_i^*$ with the nonzero coefficient $\alpha_j+1$, and
they all annihilate $\alpha+\mathbf 1$. They therefore form a basis of the hyperplane
$\{p\in N_\KK\mid p(\alpha+\mathbf 1)=0\}$, which proves the claim. Together with
$\mathfrak g_0=\mathfrak{gl}_n(\KK)\subseteq\mathfrak h$ and with $\mathfrak g_\alpha\subseteq\mathfrak h$ for
$\alpha$ outer, this gives the announced description of $\mathfrak h$, and the codimension
count follows from \cref{lem:anick-roots}\,\textup{(i)}, since $\divi$ cuts out exactly one
dimension in each $\mathfrak g_\alpha$ with $\alpha$ inner.

Finally we prove $G^{\jac}\subseteq H$. Let $\varphi\in G^{\jac}$. As $\GL_n(\KK)\subseteq H$ and $G=U_1\rtimes G_1$ with $G_1=L\cong\GL_n(\KK)$ by \cref{prop:semidirect} and \cref{lem:anick-roots}\,\textup{(ii)}, we may multiply $\varphi$ by an element of $H$ and assume $\varphi\in U_1$. Then
\[F_\varphi=(x_1+(\text{terms of order }\geq 2),\dots,x_n+(\text{terms of order }\geq 2)),\]
so $\jac(\varphi)=1$ since it is constant. The group $U_1$ is connected and unipotent by \cref{prop:semidirect}, so $\varphi=\exp(\overline{\partial})$ for a unique $\overline{\partial}\in\operatorname{Lie}(U_1)$. We show by descending induction on the
degree $k\geq1$ of the lowest nonzero homogeneous component $\overline{\partial}_k$ of $\overline{\partial}$ that
$\varphi\in H$. When $\overline{\partial}=0$ there is nothing to prove. By
\cref{lem:anick-jac}\,\textup{(ii)} and $\jac(\varphi)=1$ we get $\divi(\overline{\partial}_k)=0$, so
$\overline{\partial}_k\in\mathfrak h$ and $\exp(-\overline{\partial}_k)\in H$. Put $\varphi'=\exp(-\overline{\partial}_k)\varphi$. Then
$\varphi'\in U_1\cap G^{\jac}$, so $\jac(\varphi')=1$ as before. Write $\varphi'=\exp(\overline{\partial}')$ with $\overline{\partial}'\in\operatorname{Lie}(U_1)$. Since $\exp(-\overline{\partial}_k)=\mathrm{Id}-\overline{\partial}_k+(\text{deg}> 2k)$ and $\exp(\overline{\partial})=\mathrm{Id}+\overline{\partial}_k+(\text{deg}>k+1)$, composing gives
\[F_{\varphi'}=(x_1+(\text{terms of order }> k+1),\dots,x_n+(\text{terms of order }> k+1)),\]
so the lowest-degree component of $\overline{\partial}'$ has degree $>k$. By induction $\varphi'\in H$, hence $\varphi\in H$.

\end{proof}

\begin{remark}\label{rmk:anick-jac-needed}
For $d\geq3$ the inclusion $H\subseteq G$ is strict, so the hypothesis on the Jacobian
determinant in \cref{thm:anick} below cannot be dropped. For instance, for $n=2$ and $d=3$
the automorphism of $\KK[x_1,x_2]/\M^3$ given by
$\overline x_1\mapsto\overline x_1+\overline x_1^{\,2}$ and
$\overline x_2\mapsto\overline x_2$ is $\exp(\overline\partial_{e_1,e_1^*})$, and
$\divi(\overline\partial_{e_1,e_1^*})=2x_1\neq0$ by \eqref{eq:divergence}; its Jacobian
determinant is $1+2x_1$, which is not constant. Accordingly, its natural lift
$x_1\mapsto x_1+x_1^2$ is not an automorphism of $\KK[x_1,x_2]$.
\end{remark}

\begin{theorem}[Anick, {\cite[Theorem~1]{Anick83}}]\label{thm:anick}
Let $F=(F_1,\dots,F_n)$ be an endomorphism of $\KK[\XX]$ whose Jacobian determinant
$\det\bigl(\partial F_i/\partial x_j\bigr)$ lies in $\KK^\times$. Then for every $d\geq1$
there exists a tame automorphism $\Phi$ of $\KK[\XX]$ such that
$F_i\equiv\Phi_i\pmod{\M^d}$ for every $i$. In other words, the tame automorphisms are
dense, for the $\M$-adic topology, in the set of endomorphisms with constant nonzero
Jacobian determinant.
\end{theorem}

\begin{proof}
For $n=1$ the hypothesis says that $F_1'$ is a nonzero constant, so $F_1$ is affine and
already tame. Assume then $n\geq2$; we may also assume $d\geq2$. Composing with the
translation $x_i\mapsto x_i-F_i(0)$, which is affine and hence tame, we may assume
$F(0)=0$, that is, $F_i\in\M$ for every $i$.

Then $F$ maps $\M^d$ into itself and therefore induces an endomorphism $\varphi$ of
$A_d=\KK[\XX]/\M^d$. Since $\det J_F$ is a nonzero constant, the matrix $J_F(0)$ is
invertible and the formal inverse function theorem \cite[Theorem 1.1.2]{vdE00} provides
$\Psi\in\KK[[\XX]]^n$ with $\Psi(0)=0$ and $F\circ\Psi=\Psi\circ F=\mathrm{id}$. Truncating
$\Psi$ in degrees $<d$ gives an endomorphism of $A_d$ which is a two-sided inverse of
$\varphi$, so $\varphi\in G$.

The tuple $F_\varphi$ is the truncation of $F$ in degrees $<d$, whose Jacobian matrix
differs from $J_F$ by entries in $\M^{d-1}$; hence
$\jac(\varphi)=\det J_F\in\KK^\times$ and $\varphi\in G^{\jac}$. By
\cref{prop:anick-outer}, $\varphi$ is a product of elements of $T$ and of the outer root
subgroups $U_\alpha$, and of their inverses.

Each of these lifts to a tame automorphism of $\KK[\XX]$ fixing the origin: an element of $T$ lifts to the diagonal automorphism $x_i\mapsto t_ix_i$, which is affine; and, as in the proof of \cref{prop:anick-outer}, the element $\exp(t\,\overline\partial_{\alpha,e_j^*})$ of $U_\alpha$, with $\alpha\in\R_j(\M^d)$, lifts to
\[x_j\mapsto x_j+t\XX^{\alpha+e_j},\qquad x_i\mapsto x_i,\,\text{ for } i\neq j,\]
which is elementary because $(\alpha+e_j)_j=0$. Since reduction modulo $\M^d$ is a group homomorphism from the automorphisms of $\KK[\XX]$ fixing the origin to $G$, the corresponding product $\Phi$ of these tame automorphisms satisfies $\overline\Phi=\varphi$, that is, $F_i\equiv\Phi_i\pmod{\M^d}$ for every $i$.

\end{proof}

\bibliographystyle{alpha}
\bibliography{ref}
\end{document}